\documentclass{article}

\usepackage[english]{babel}

\usepackage[letterpaper,top=2cm,bottom=2cm,left=3cm,right=3cm,marginparwidth=1.75cm]{geometry}

\usepackage{amsmath}
\usepackage{amsthm}
\usepackage{amsfonts}
\usepackage{amssymb}
\usepackage{graphicx}
\usepackage{tikz}
\usetikzlibrary{decorations.pathreplacing}
\usepackage{subcaption}
\usetikzlibrary{arrows}
\usetikzlibrary{arrows.meta}
\usepackage{algorithm}
\usepackage{algpseudocode}
\usepackage{cases}
\usepackage{titlesec}
\usepackage{subcaption}
\usepackage{mathtools}
\usepackage{todonotes}
\usepackage{changes}
\usepackage{booktabs}
\usepackage{thm-restate}
\usepackage{url}
\usepackage{authblk}
\usepackage[short]{optidef}
\usepackage[colorlinks=true, allcolors=blue]{hyperref}
\usepackage{cleveref}
\definecolor{darkgreen}{rgb}{0,0.5,0}
\usepackage{accents}
\MakeRobust{\underaccent}
\newcommand{\ubar}[1]{\underaccent{\bar}{#1}}

\newtheorem{theorem}{Theorem}
\newtheorem{corollary}[theorem]{Corollary}
\newtheorem{definition}[theorem]{Definition}
\newtheorem{example}[theorem]{Example}
\newtheorem{remark}[theorem]{Remark}
\newtheorem*{remark*}{Remark}
\newtheorem{lemma}[theorem]{Lemma}
\newtheorem{problem}[theorem]{Problem}

\newcommand{\MA}{\emph{MA}}
\newcommand{\LA}{\emph{LA}}
\newcommand{\cUint}{\mathcal{U}^\square}

\usepackage[
    backend=biber,
    style=numeric-comp,
    ]{biblatex}
\title{A robust optimization approach to flow decomposition}
\author[1,2]{Moritz Stinzendörfer}
\author[1]{Philine Schiewe}
\author[1,3]{Fabricio Oliveira\thanks{Corresponding author: fabol@dtu.dk}}

\affil[1]{Department of Mathematics and Systems Analysis, Aalto University, Otakaari 1, 02150 Espoo, Finland}
\affil[2]{Department of Mathematics, RPTU Kaiserslautern-Landau, Paul-Ehrlich-Str.~14, 67663 Kaiserslautern, Germany}
\affil[3]{DTU Management, Technical University of Denmark, Akademivej 358,
2800 Kgs. Lyngby, Denmark}

\begin{document}
\maketitle

\begin{abstract}
In this paper, we generalize the minimum flow decomposition problem (MFD) to incorporate uncertain edge capacities and tackle it from the perspective of robust optimization.
In the classical flow decomposition problem, a network flow is decomposed into a set of weighted paths from a fixed source node to a fixed sink node that precisely represents the flow distribution across all edges.
MFD problems permeate multiple important applications, including reconstructing genomic sequences to representing the flow of goods or passengers in distribution networks. Inspired by these applications, we generalize the MFD to an inexact case with bounded flow values, provide a detailed analysis, and explore different variants that are solvable in polynomial time. Moreover, we introduce the concept of robust flow decomposition by incorporating uncertain bounds and applying different robustness concepts to handle the uncertainty. Finally, we present two different adjustably robust problem formulations and perform computational experiments illustrating the benefit of adjustability.
\end{abstract}


\section{Introduction} \label{sec:introduction}

Network flows are a fundamental concept in mathematics, particularly in the field of graph theory and optimization. A network is typically represented as a directed graph, with nodes representing junctions or points of interest and edges representing the pathways or connections between these nodes.
The primary goal in network flow problems is to determine an optimal way to send a certain amount of flow from a source node to a sink node or find a minimum cost flow. In the latter, each edge has a cost associated with transporting flow through it, and the objective is to minimize the total cost while satisfying flow requirements. These network flow problems have wide-ranging applications in fields such as transportation \cite{bunte2009overview}, telecommunications \cite{li2013survey}, logistics \cite{salimifard2022multicommodity}, and bioinformatics \cite{williams2019rna}, making them an essential area of study in both theoretical and applied mathematics.

Another important network flow problem is the \textit{minimum flow decomposition} (MFD), which involves decomposing the flow into a set of weighted paths that, once combined, recover how the flow is distributed through the network. This problem structure is present in bioinformatics applications such as multiassembly problems \cite{dias2022efficient}, which involve reconstructing genomic sequences from short substrings. For example, in RNA transcript assembly, sequenced substrings from a given sample are used to construct weighted graphs (known as splice or splicing graphs) where nodes represent exons (i.e., the genetic information that remains in the messenger RNA) and edges represent connections between exons. From the splice graph, one seeks to infer how many transcripts (sequences of exons) there are in the splice graph (i.e., weighted paths) such that their superposition recovers (or explains) the given splice graph.
However, real-world data often is marred with imprecision and measurement errors, making it challenging to form accurate flow networks, which is why the related research focuses on robustness against measurement errors \cite{dias2023accurate}.

Another sector where MFDs are also prevalent is transportation planning. In particular, 
public transport planning involves several stages, including line planning, timetabling, and vehicle scheduling. The latter stage, in particular, can benefit from the MFD by decomposing trip graphs (i.e., graphs representing the aggregated number of trips between nodes that represent locations) into a minimal number of paths, thereby optimizing vehicle utilization and reducing operational costs.
Despite the extensive literature on robustness considerations in public transport planning, most research focuses on timetabling and delay management, considering uncertainties such as travel times or disruptions. On the other hand, demand uncertainties are often addressed in applications focusing on more operational levels, such as vehicle routing problems. 

The above applications share as a common feature the latent need of being able to incorporate robustness requirements in their modeling, so that their recommendations are reliable in the face of the uncertainty in their input data. Aiming to address this need, in this paper, we explore the minimum flow decomposition problem (MFDP) from a robust optimization perspective and propose the notion of \emph{robust flow decomposition}. To the best of our knowledge, this is the first work in this direction.

Our main contributions can be summarized as follows:
\begin{itemize}
    \item we examine existing variants of MFDP with regard to different robustness concepts,
    \item we generalize the MFDP to the weighted inexact case with lower and upper bounds on the flow values, investigate the complexity of the resulting problem, and explore different variants that are solvable in polynomial time,
    \item we introduce the concept of robust flow decomposition by incorporating uncertain flows and discuss special cases with different robustness concepts,
    \item and we present two different adjustable problem formulations for which we develop a proof of concept, highlighting the benefit of adjustability in the uncertain case in a subsequent computational study.
\end{itemize}

The remainder of this paper is structured as follows: The connections to the related literature are discussed in Section \ref{sec:literatur}, while in Section~\ref{sec:def} we present the underlying technical background concerning the MFDP.
In Section \ref{sec:model}, we introduce the generalized deterministic problem considering inexact flows and provide complexity analyses of relevant variants. In Section~\ref{sec:robust_flows}, we develop robust minimum flow decomposition problems, starting with strict robustness in Section~\ref{sec:robustMFD}. 
Further, we propose two adjustable problem formulations in Section \ref{sec:adj} followed by a corresponding computational study in Section \ref{sec:study}.
Section \ref{sec:concl} concludes the paper.

\section{Literature review}\label{sec:literatur}

The task of optimally covering nodes, edges, or paths in a graph is a common problem that appears in many well-studied concepts in areas of graph theory and network analysis. One example is to find a \emph{minimum path cover (MPC)}, i.e., a set of directed paths with minimum cardinality that covers all vertices (or edges), in a digraph \cite{diestel2005graph}. Since the MPC consists of one path if and only if there is a Hamiltonian path in the corresponding graph, the MPC problem is NP-hard.

However, unlike the MFDP and its robust variants we investigate, the edges only need to be covered by a minimum number of paths, and there are no flow values that have to match weights associated with the paths.
Therefore, the MPC can be solved in polynomial time if the graph is acyclic, e.g., by transforming it into a matching problem \cite{ntafos1979path}.

The \emph{flow decomposition problem (FDP)} \cite{ahuja1988network}, on the other hand, describes the problem of decomposing a network flow into a set of weighted paths from a fixed source node to a fixed sink node such that their union precisely represents the flow distribution across all edges.
One way to compute such a flow decomposition in a \emph{directed acyclic graph (DAG)} is to iteratively remove weighted paths that utilize at least one edge.
The authors of \cite{ahuja1988network} show that this results in a flow decomposition with at most $|E|$ paths (where $|E|$ is the number of edges in the graph) and can be computed in polynomial time.
Nonetheless, when aiming to minimize the number of paths in the decomposition, referred to as \emph{minimum flow decomposition (MFD)}, the problem becomes NP-hard \cite{vatinlen2008simple} and is even hard to approximate \cite{hartman2012split}.
As a result, efficient heuristics have been developed, such as the greedy methods in \cite{vatinlen2008simple} or an improved version in \cite{shao2017theory}.
In the former, the proposed algorithms iteratively choose the shortest path or the path with the largest flow value in the remaining flow until the entire flow is decomposed, while in the latter, the flow graph is first modified before applying the greedy method.


Recently, the authors of \cite{dias2022efficient} presented an exact solution approach for the MFDP on DAGs, which is based on an integer linear programming (ILP) approach using only a quadratic number of variables instead of enumerating all possible paths in the network (which would lead to an exponential number of variables).
They show that their proposed method consistently solves instances on both simulated and real datasets significantly faster than previous approaches. These data sets originate from applications in bioinformatics, such as \emph{multiassembly problems}, where MFD plays a key role \cite{xing2004multiassembly}.

\subsection{Inexact flow decomposition in transcript assembly}

The multiassembly problem involves reconstructing various genomic sequences from short substrings (called sequenced \emph{reads}) \cite{xing2004multiassembly}, such as RNA transcript assembly \cite{li2011isolasso,shao2017theory}, i.e., recovering the set of full-length transcripts.
Initially, reads are used to construct a weighted graph, namely a splice graph in RNA transcript assembly, where nodes represent exons, and edges denote connections between them.
The nodes and edges are weighted, indicating, e.g., their relative abundance in the reads.
The multiassembly problem is then finding a set of weighted paths that best explain the graph's weights \cite{tomescu2015explaining,shao2017theory}, and it has been shown that the MFD has a very high accuracy on perfect data \cite{dias2023accurate}.

However, the assumption of perfect data does not hold in practice. Due to different sources of uncertainty and measurement errors, splice graphs derived from experimental data are unlikely to form a flow network \cite{williams2019rna}.
To address this problem, the authors of \cite{williams2019rna} assign intervals of possible weights to the edges instead of exact weights, which forms a so-called \emph{inexact flow network}.
The resulting problem of identifying a minimal set of paths that explain these intervals is called the \emph{minimum inexact flow decomposition problem (MIFDP)}, which is also considered in \cite{dias2022efficient}. Other approaches aim to identify a minimum set of weighted paths that minimizes the sum of squared differences between the weight of each edge and the sum of weights of paths passing through it or to assign slack variables to each path to move error handling away from the individual edges \cite{dias2023accurate}.

All these approaches have in common that they try to be robust against possible measurement errors. This means that although all data is known, it is likely to be error-prone, which is why, in most cases, no classical flow decomposition can be achieved. However, to ensure that the paths in the decomposition form feasible flows, the requirement that the measured weight of each edge and the sum of weights of paths passing through it must match is relaxed.
In contrast to the typical assumptions in robust optimization, a scenario or the real data is not revealed here. While this is suitable in the context of bioinformatics applications, robustness in the classical sense of finding a feasible solution that is robust for every possible realization in an uncertainty set is better suited to enforce robustness requirements in transportation problems or public transport planning, where the MFD also plays a key role.

\subsection{Inexact flow decomposition in public transportation planning}

In public transport planning, the operational costs are mainly determined by the \emph{vehicle schedule} \cite{schiewe2022integrated}, which is part of the traditional sequential planning approach \cite{huisman2005operations,guihaire2008transit}.
Typically, this begins on a strategic level with a demand analysis to build passenger demand matrices, which are then used to define the overall structure of the transit network, including necessary infrastructure such as stations or depots.

In the next step, called \emph{line planning}, the types of services (e.g., bus, rail, tram) and their routes, stops, and frequencies are determined to meet the demand and policy goals \cite{schobel2012line}.
This is followed by the operational planning stage, where the \emph{timetabling} step focuses on creating detailed schedules for each line, resulting in so-called \emph{trips} with departure and arrival times for each station as well as start and end stations \cite{lusby2011railway}.
Then, vehicle scheduling allocates vehicles to trips, ensuring that all scheduled trips are covered while optimizing vehicle utilization, often characterized by minimal fleet size or operational costs \cite{bunte2009overview}.
The former can be addressed by decomposing the trip graph into a minimal number of paths \cite{bunte2009overview}, and also in recent applications such as electric buses with limited driving ranges, decomposition methods are used for vehicle scheduling \cite{olsen2022study}.

Although there is a large body of literature dedicated to robustness in public transport, most of it concentrates on timetabling and delay management by considering uncertain travel times or disruptions \cite{cacchiani2012nominal,parbo2016passenger,lusby2018survey}. 
Demand uncertainties are often considered in earlier stages, e.g., stop planning \cite{cacchiani2020robust}, or in other applications such as \emph{vehicle routing problems} \cite{lee2012robust}.

\subsection{Robust optimization}
The concept of \emph{robust optimization} was introduced by \cite{soyster1973convex,ben1998robust}, which usually refers to finding an optimal solution that is feasible for all possible realizations (often represented as scenarios) induced by an uncertainty set.
For an extensive overview of theoretical properties and applications, we refer the reader to \cite{beyer2007robust,bertsimas2011theory,gabrel2014recent}.
Instead of finding the best robust solution for the worst-case scenario, in \emph{regret robustness} \cite{kouvelis2013robust}, the objective is to minimize the difference (\emph{regret}) between the robust solution value and the best objective value we could have achieved if the realization had been known beforehand.

Since these approaches still require the solution to be feasible for all possible scenarios, the concept of \emph{light robustness} \cite{schobel2014generalized} weakens the feasibility constraint by only searching for solutions that are feasible and ``good enough'' in the \emph{nominal case} (e.g., the most likely case).
An alternative idea to overcome over-conservatism is to bound the uncertainty set by parameterizing the allowed cumulative deviation from the nominal case, resulting in a \emph{budgeted uncertainty set} \cite{bertsimas2004price}.
A similar approach is followed for the \emph{optimization problems under controllable uncertainty} \cite{ley2023robust}, where the uncertainty is allowed to be reduced at a given cost.

The concept \emph{adjustable robustness} \cite{ben2004adjustable} takes into account that, in many applications, it is possible to adjust a part of the decisions after the uncertainty is unveiled.
The idea is to divide the variables into those that must be decided before the uncertainty is unveiled and those that can be decided after its realization.
For an overview on adjustable robust optimization, we refer the reader to \cite{yanikouglu2019survey}.

As previous concepts mainly concentrate on applications in bioinformatics, the focus to date has primarily been on dealing with data inaccuracies. To the best of our knowledge, this is the first work that considers MFDs in the context of classical robustness concepts. This opens up many other application possibilities that are typically represented using flow networks and in which uncertainty plays a role. In public transport, for example, it is now also possible to take demand uncertainties into account in later planning steps, such as vehicle scheduling. In general, this allows robust and, depending on the application, resource-efficient solutions to be achieved despite uncertain data.

\section{Minimum flow decomposition}\label{sec:def}



Let $G(V,E)$ be a directed acyclic graph (DAG), where $s \in V$ represents the \emph{source} node with no incoming edges and $t \in V$ the \emph{sink} (\emph{target}) node with no outgoing edges.
Let $f_{uv}$ be the corresponding non-negative integer flow value of edge $(u,v)$ for all $(u,v) \in E$, i.e., $f: E \rightarrow \mathbb{N}_{\geq 0}$.
Definition \ref{def:network_flow} provides a formal description of a \emph{flow network}. An example of a flow network is given in Figure \ref{fig:examplea}.
\begin{definition}[Flow network] \label{def:network_flow}
    The tuple $G(V,E,f)$ is called a \emph{flow network} if for every $v \in V \setminus \{s,t\}$ the conservation of flow 
    \begin{equation}
        \sum_{u:(u,v) \in E } f_{uv} = \sum_{w:(v,w) \in E } f_{vw}
    \end{equation}
    is satisfied.
\end{definition}

Given such a flow network, it can be decomposed into a set of \emph{$s$-$t$-paths}, where each path has an associated positive weight (Figure \ref{fig:exampleb}).
Definition \ref{def:k-flow_decomposition} describes how $s$-$t$-paths form a $k$-flow decomposition.

\begin{definition}[$k$-flow decomposition]\label{def:k-flow_decomposition}
    For a given flow network $G(V,E,f)$, a set of $s$-$t$-paths $\mathcal{P} = (P_1,...,P_k)$ with corresponding positive weights $w = (w_1,...,w_k)$ is called \emph{$k$-flow decomposition} if
    \begin{equation}
        \sum_{i\colon (u,v) \in P_i} w_i = f_{uv}
    \end{equation}
    holds for all $(u,v) \in E$. 
\end{definition}

\begin{remark} \label{def:netflow_decomposition}
Let us define $| f |$ of a flow $f$ is the net flow $\sum_{u:(u,t) \in E } f_{ut}$ into the sink node $t$ \cite{goldberg1989network}. As we assume integer flow values, the trivial $| f |$-flow decomposition with $| f |$ paths of weight one \cite{ahuja1988network} provides an upper bound for $k$.
\end{remark}

\begin{problem}[Minimum flow decomposition problem] \label{problem:MFD}
    For a given flow network $G(V,E,f)$, the \emph{minimum flow decomposition problem (MFDP)} is to find a $k$-flow decomposition $\mathcal{P} = (P_1,...,P_k)$ such that $k$ is minimized.
\end{problem}

Example \ref{ex:mfd} provides an illustration of Problem \ref{problem:MFD}, including examples of a given network flow (cf. Definition \ref{def:network_flow}) and a feasible 5-flow decomposition (cf. Definition \ref{def:k-flow_decomposition}).

\begin{example}[\cite{dias2023accurate}]\label{ex:mfd}
    Assume we are given the graph in Figure \ref{fig:examplea} with flow values on the corresponding edges.
    Then this flow network can be decomposed into the five colored paths $P_1, P_2, P_3, P_4, P_5$ shown in Figure \ref{fig:exampleb} with the associated weights $w_1=1,w_2=2,w_3=2,w_4=2,w_5=3$.
    Moreover, this is an optimal solution to MFDP with a solution value of $k=5$.

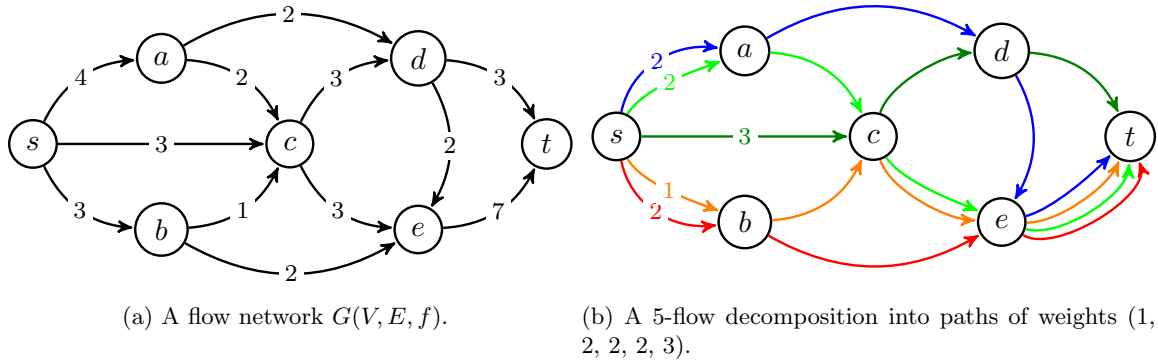
\begin{figure}[h]
	\begin{center}
        \begin{subfigure}{0.49\linewidth}
	   \resizebox{1\textwidth}{!}{
			\begin{tikzpicture}[->,>=stealth',shorten >=0.5pt,node distance=1.7cm, thick,main node/.style={circle,draw,font=\sffamily\small}]
   \tikzset{edgelabel/.style={fill=white,font=\footnotesize,inner sep=2pt}}
			\begin{scope}[every node/.style={circle,thick,draw}]
                \node (0) at (0,0) {$ s $};
                \node (1) at (1.5,1) {$ a $};
			    \node (2) at (1.5,-1) {$ b $};
			    \node (3) at (3,0) {$ c $};
			    \node (4) at (4.5,1) {$ d $};
			    \node (5) at (4.5,-1) {$ e $};
                \node (6) at (6,0) {$ t $};
		    \end{scope}
		    \path[every node/.style={font=\sffamily\small}]
    				(0) edge [bend left] node [midway,edgelabel] 	{$ 4 $} 		(1)
    					edge [bend right] node [midway,edgelabel] 	    {$ 3 $} 		(2)
                        edge node [midway,edgelabel] 	    {$ 3 $} 		(3)
					(1) edge [bend left] node [midway,edgelabel] 	{$ 2 $} 					(3)
						edge [bend left] node [ midway,edgelabel] 	{$ 2 $} 		(4)
					(2) edge [bend right] node [midway,edgelabel] 	{$ 1 $} 					(3)
                        edge [bend right] node [midway,edgelabel] 	{$ 2 $} 					(5)
					(3) edge [bend left] node [midway,edgelabel] 	{$ 3 $} 					(4)
                        edge [bend right] node [midway,edgelabel] 	{$ 3 $} 					(5)
					(4) edge [bend left] node [midway,edgelabel] 	{$ 2 $} 					(5)
						edge [bend left] node [midway,edgelabel] 	{$ 3 $} 		(6)
					(5) edge [bend right] node [midway,edgelabel] 	{$ 7 $} 					(6);
					
		\end{tikzpicture}}
        \caption{A flow network $G(V,E,f)$.\\~}
        \label{fig:examplea}
\end{subfigure}
\begin{subfigure}{0.49\linewidth}
	\resizebox{1\textwidth}{!}{
			\begin{tikzpicture}[->,>=stealth',shorten >=0.5pt,node distance=1.7cm, thick,main node/.style={circle,draw,font=\sffamily\small}]
   \tikzset{edgelabel/.style={fill=white,font=\footnotesize,inner sep=2pt}}
			\begin{scope}[every node/.style={circle,thick,draw}]
                \node (0) at (0,0) {$ s $};
                \node (1) at (1.5,1) {$ a $};
			    \node (2) at (1.5,-1) {$ b $};
			    \node (3) at (3,0) {$ c $};
			    \node (4) at (4.5,1) {$ d $};
			    \node (5) at (4.5,-1) {$ e $};
                \node (6) at (6,0) {$ t $};
		    \end{scope}
      
		    \path[every node/.style={font=\sffamily\small}]
    				(0) edge [bend left, blue, out=45, in=140, looseness=1.0] node [midway,edgelabel] 	{$ 2 $} (1)
                        edge [bend left, green, out=30, in=170, looseness=0.7] node [midway,edgelabel] 	{$ 2 $} (1)
    					edge [bend left, red, out=315, in=-140, looseness=1.0] node [midway,edgelabel] 	{$ 2 $} (2)
                        edge [bend left, orange, out=340, in=-170, looseness=0.7] node [midway,edgelabel] 	{$ 1 $} (2)
                        edge [darkgreen] node [midway,edgelabel] 	{$ 3 $}	(3)
					(1) edge [bend left, green] node [left] 	{} 					(3)
						edge [bend left,blue] node [ above ] 	{} 		(4)
					(2) edge [bend right, orange] node [ below ] 	{} 					(3)
                        edge [bend right, red] node [ below ] 	{} 					(5)
					(3) 
                        edge [bend left, darkgreen, out=-320, in=150, looseness=0.7] node [] 	{} (4)
                        edge [bend left, orange, out=320, in=-150, looseness=0.7] node [] 	{} (5)
                        edge [bend left, green, out=335, in=-172, looseness=0.7] node [] 	{} (5)
					(4) edge [bend left, blue] node [ right ] 	{} 					(5)
						edge [bend left, darkgreen] node [ above ] 	{} 		(6)
					(5) edge [bend left, blue, out=340, in=-170, looseness=0.7] node [] 	{} (6)
                        edge [bend left, orange, out=325, in=-148, looseness=0.7] node [] 	{} (6)
                        edge [bend left, green, out=310, in=-126, looseness=1.0] node [] 	{} (6)
                        edge [bend left, red, out=295, in=-104, looseness=1.1] node [] 	{} (6);

		\end{tikzpicture}}
        \caption{A 5-flow decomposition into paths of weights (1, 2, 2, 2, 3).}
        \label{fig:exampleb}
		\end{subfigure}
\end{center}
	\caption{Example of a flow network and a flow decomposition into 5 $s$-$t$-paths \cite{dias2023accurate}.}
	\label{fig:example}
	\end{figure}	
\end{example}

For a given upper bound $\overline{K}$ of $k$ (for example, $| f |$, cf. Remark \ref{def:k-flow_decomposition}), Problem \ref{problem:MFD} can be posed as the corresponding MIP formulation\footnote{Note that \eqref{eq:MFD} is formally an MIP only if we linearize constraints \eqref{eq:MFD4}, which can be done using standard linear reformulation techniques (see, e.g., \cite{dias2022efficient})}.
\begin{mini!}
{x,y,w}{\sum_{i=1}^{\overline{K}} y_i \label{eq:ObjectiveExample1}}
{\label{eq:MFD}}{}
\addConstraint{\sum_{v:(s,v) \in E} x_{svi}}{=y_i, \label{eq:MFD1}}{ \forall\; i \in \{1,...,\overline{K}\}}
\addConstraint{\sum_{u:(u,t) \in E} x_{uti}}{=y_i, \label{eq:MFD2}}{ \forall\; i \in \{1,...,\overline{K}\}}
\addConstraint{\sum_{(u,v) \in E} x_{uvi} - \sum_{(v,w) \in E} x_{vwi}}{=0, \label{eq:MFD3}}{ \forall\; i \in \{1,...,\overline{K}\}, v \in V \setminus \{s,t\}}
\addConstraint{\sum_{i \in \{1,...,\overline{K}\}} w_i x_{uvi}}{= f_{uv}, \label{eq:MFD4}}{\forall\; (u,v) \in E}
\addConstraint{w_i}{\in \mathbb{Z}^+,  \label{eq:MFD5}}{ \forall\; i \in \{1,...,\overline{K}\}}
\addConstraint{x_{uvi}}{\in \{0,1\}, \quad \label{eq:MFD6}}{\forall\; (u,v) \in E, i \in \{1,...,\overline{K}\}}
\addConstraint{y_i}{\in \{0,1\}, \label{eq:MFD7}}{\forall\; i \in \{1,...,\overline{K}\}.}
\end{mini!}    

Formulation \eqref{eq:MFD} is essentially the formulation in \cite{dias2022efficient}. Here, the binary variable $y_i$ indicates whether the path $i$ with the corresponding weight $w_i \in \mathbb{Z}^+$ is part of the flow decomposition ($y_i=1$) or not ($y_i=0$). The variable $x_{uvi} \in \{0,1\}$ is equal to one if edge $(u,v) \in E $ is part of path $i \in \{1,...,\overline{K}\}$. Every path $i$ starts at node $s$ (i.e., has exactly one variable $x_{svi}$ set to 1 for some edge $(s, v) \in E$) and ends at node $t$ (has exactly one variable $x_{vti}$ set to 1, for some edge $(v, t) \in E$). This is enforced by \eqref{eq:MFD1} and \eqref{eq:MFD2}, respectively. Constraint \eqref{eq:MFD3}, together with \eqref{eq:MFD1} and \eqref{eq:MFD2}, characterize a path from $s$ to $t$. Finally, \eqref{eq:MFD4} enforces that the combined weights $w_i$ of all paths $i$ that traverse the edge $(u,v)$ match the flow $f_{u,v}$. Constraints \eqref{eq:MFD5}-\eqref{eq:MFD7} define variable domains. 
Note that one can straightforwardly add symmetry-breaking constraints to \eqref{eq:MWIFD}, e.g., by imposing ordering constraints on the variables $y$ and $w$. 

The MFDP can be generalized by relaxing constraints \eqref{eq:MFD4}, allowing the summed path weights of each edge $(u,v)$ to lie in a range $[f^l_{uv},f^u_{uv}]$ instead of matching a flow value $f_{uv}$.
Consequently, the underlying graph is now an \emph{inexact flow network} \cite{williams2019rna}, meaning that we have a lower bound $f^l_{uv}$ and an upper bound flow value $f^u_{uv}$ corresponding to each edge $(u,v) \in E$.
The resulting problem, the minimum inexact flow decomposition problem (MIFDP) is summarized in Problem \ref{prob:MIDP}. Note that both MFDP and its generalization, MIFDP, are NP-hard \cite{dias2023accurate}.


\begin{problem}[Minimum inexact flow decomposition problem \cite{williams2019rna}] \label{prob:MIDP}
    For a given inexact flow network $G(V,E,f^l,f^u)$, the minimum inexact flow decomposition problem (MIFDP) is to find a set of $s$-$t$-paths $\mathcal{P} = (P_1,...,P_k)$ with minimum cardinality and associated positive weights $w = (w_1,...,w_k)$ such that
    \begin{equation}
        f^l_{uv} \leq \sum_{i: (u,v) \in P_i} w_i \leq f^u_{uv}
    \end{equation}
    holds for all $(u,v) \in E$. 
    We obtain its corresponding MIP formulation by replacing constraints \eqref{eq:MFD4} in \eqref{eq:MFD} with \eqref{eq:MIFD}
    \begin{equation}\label{eq:MIFD}
        \begin{aligned}
    		&& \sum_{i \in \{1,...,\overline{K}\}} w_i x_{uvi}                   &\geq f^l_{uv} &&  \forall\; (u,v) \in E,   \\
            && \sum_{i \in \{1,...,\overline{K}\}} w_i x_{uvi}                   &\leq f^u_{uv} &&  \forall\; (u,v) \in E.
        \end{aligned}
    \end{equation} 
\end{problem}%

\section{Minimum weighted inexact flow decomposition problem}\label{sec:model}
In this section, we introduce a generalized version of the MIFDP by minimizing the weighted sum of the associated weights $a_w \sum_{i=1}^{|\mathcal{P}|} w_i$ in addition to the number of paths $a_y |\mathcal{P}|$.
Their presence in the objective function with corresponding weights $a_y, a_w \in \mathbb{R}_{\geq0} $ motivates the analysis of the subproblems resulting from different weight combinations.
While $a_y > 0 $, $a_w = 0$ corresponds to the MIFDP, we show how $a_y = 0 $, $a_w > 0$ can be solved in polynomial time.
Moreover, polynomial time can also be achieved for $a_y > 0 $, $a_w = 0$ if the upper bounds on the flow values are neglected.

\subsection{General formulation}\label{subsec:generalformulation}
Using the MIFDP as a starting point, the weight variables $w$ no longer necessarily match the flow values since the summed path weights of each edge $(u,v)$ lie in a range $[f^l_{uv},f^u_{uv}]$ instead.
Depending on the magnitude of the upper bounds in $f^u$, the entries of $w$ can take on large values, even if smaller ones would be sufficient to cover the corresponding lower bounds.
Therefore, it is reasonable to extend the objective function by the sum of all weights, as in many related applications, the weights are desired to be as small as possible, e.g., if they correspond to a vehicle's capacity.
The resulting problem is formalized next.
\begin{problem}[Minimum weighted inexact flow decomposition problem]
    For a given inexact flow network $G(V,E,f^l,f^u)$, the \emph{minimum weighted inexact flow decomposition problem (MWIFDP)} is to find a set of $s$-$t$-paths $\mathcal{P} = (P_1,...,P_k)$ with associated positive weights $w = (w_1,...,w_k)$ such that
    \begin{equation*}
        f^l_{uv} \leq \sum_{i: (u,v) \in P_i} w_i \leq f^u_{uv}
    \end{equation*}
    holds for all $(u,v) \in E$ and the weighted sum 
    \begin{equation}
        a_y |\mathcal{P}| + a_w \sum_{i=1}^{|\mathcal{P}|} w_i ,
    \end{equation}
    with weights $a_y, a_w \in \mathbb{R}_{\geq0} $ is minimized.
\end{problem}

For a given upper bound $\overline{K}$ of $k$, the corresponding MIP formulation for the MWIFDP is given by
\begin{mini!}
{x,y,w}{\sum_{i=1}^{\overline{K}} (a_y y_i + a_w w_i) \label{eq:MWIFD0}}
{\label{eq:MWIFD}}{}
\addConstraint{\sum_{v:(s,v) \in E} x_{svi}}{=y_i, \label{eq:MWIFD1}}{ \forall\; i \in \{1,...,\overline{K}\}}
\addConstraint{\sum_{u:(u,t) \in E} x_{uti}}{=y_i, \label{eq:MWIFD2}}{ \forall\; i \in \{1,...,\overline{K}\}}
\addConstraint{\sum_{(u,v) \in E} x_{uvi} - \sum_{(v,w) \in E} x_{vwi}}{=0, \label{eq:MWIFD3}}{ \forall\; i \in \{1,...,\overline{K}\}, v \in V \setminus \{s,t\}}
\addConstraint{\sum_{i \in \{1,...,\overline{K}\}} w_i x_{uvi}}{\geq f^l_{uv}, \label{eq:MWIFD4}}{\forall\; (u,v) \in E}
\addConstraint{\sum_{i \in \{1,...,\overline{K}\}} w_i x_{uvi}}{\leq f^u_{uv}, \label{eq:MWIFD5}}{\forall\; (u,v) \in E}
\addConstraint{w_i}{\in \mathbb{Z}^+,  \label{eq:MWIFD6}}{ \forall\; i \in \{1,...,\overline{K}\}}
\addConstraint{x_{uvi}}{\in \{0,1\}, \quad \label{eq:MWIFD7}}{\forall\; (u,v) \in E, i \in \{1,...,\overline{K}\}}
\addConstraint{y_i}{\in \{0,1\}, \label{eq:MWIFD8}}{\forall\; i \in \{1,...,\overline{K}\}.}
\end{mini!}    

\begin{corollary} \label{cor:cor7}
    For $a_y > 0 $ and $a_w = 0$, MWIFDP is NP-hard.
\end{corollary}

Notice that MWIFDP reduces to MIFDP for $a_y > 0$ and $a_w = 0$ and is therefore NP-hard. In Lemma \ref{lemma:MWIFDP-NP-Hard}, we show that this is also the case for $a_y, a_w > 0$.
\begin{lemma}\label{lemma:MWIFDP-NP-Hard}
    For $a_y, a_w > 0$, MWIFDP is NP-hard since it includes the MFDP.
\end{lemma}
\begin{proof}
 Given an instance $G(V,E,f)$ of MFDP, we can construct an instance $G(V,E,f^l,f^u)$ of MWIFDP using $f^l_{uv}=f_{uv}=f^u_{uv} $ for all $(u,v) \in E$.
 Then, the summed weights of the paths are equal to the summed flow values of all outgoing edges of $s$ (the complete $s$-$t$-flow value) and build both a lower and upper bound on $\sum_{i=1}^{\overline{K}}w_i$.
 Consequently, finding a minimum flow decomposition, which is NP-hard \cite{vatinlen2008simple}, minimizes the objective.
\end{proof}

Counterintuitively, we show that the MWIFDP is not NP-hard when $a_y = 0 $ and $a_w > 0$. To show this, we first revisit the \emph{minimum-cost flow problem} \cite{orlin1988faster} in capacitated networks.

\begin{definition}[Minimum-cost flow problem \cite{orlin1988faster}]
Given a tuple $G(V,E,d,c)$, where $G(V,E)$ is a digraph, $d(v) \in \mathbb{Z}$ represents the demand ($d(v) \leq 0 $) or the supply ($d(v) > 0 $) of node $ v \in V$ and $c(u,v) \geq 0$ is the cost associated with every edge $(u,v) \in E$.
The \emph{minimum-cost flow problem (MCFP)}  is to find a flow $f \in \mathbb{N}_{\geq0}^{|E|}$ with minimum cost $\sum_{(u,v) \in E } c(u,v) \cdot f_{uv} $ such that
\begin{equation}
    \sum_{u:(u,v) \in E } f_{uv} - \sum_{w:(v,w) \in E } f_{vw} = d(v)
\end{equation}
is satisfied for all $v \in V.$
In capacitated networks, we have $G(V,E,d,c,\ubar{f},\bar{f})$ with lower $\ubar{f}_{uv}$ and upper bounds $\bar{f}_{uv}$ on the flow $f_{uv}$ associated with each edge $(u,v) \in E$, which also have to be fulfilled.
\end{definition}

\begin{lemma}\label{lemma:1}
    For $a_y = 0 $ and $a_w > 0$, MWIFDP can be solved in polynomial time.
\end{lemma}
\begin{proof}
    This particular case can be reduced to the minimum-cost flow problem in capacitated networks, which can be solved in polynomial time \cite{orlin1988faster,wayne1999polynomial,zhu2011minimal}.
    For a given inexact flow network $G(V,E,f^l,f^u)$, we show that solving the MWIFDP with $a_y = 0 $ and $a_w > 0$ corresponds to finding a minimal cost flow in the following graph. 
    
    For $G'(V',E',d,c,\ubar{f},\bar{f})$, let $V'=V \cup \{s'\} \cup \{t'\} $ be the node set with a super-source $s'$ and super-sink $t'$ and $E'=E \cup \{ (s',s), (t,t'), (s',t')\} $ the corresponding edges.
    We set the supply and demand of $s'$ and $t'$ equal to the summed upper bounds on the flow values $\sum_{(u,v) \in E} f^u_{uv} =: M$, i.e., $d(s') = M = - d(t')$, while we have $d(v)=0$ for all other nodes $v \in V$.
    For the cost function $c$, we have $c(s',s)=1$ and $c(e)=0$ for all $e \in E' \setminus \{(s',s)\}$, and the lower and upper bounds, $\ubar{f}$, $\bar{f}$, on the flow are defined as $(\ubar{f}_{uv},\bar{f}_{uv}) := (f^l_{uv},f^u_{uv})$ for $(u,v) \in E$ and $(\ubar{f}_{uv},\bar{f}_{uv}) := (0,M)$ for $ (u,v) \in \{ (s',s), (t,t'), (s',t') \}$.

    In other words, we extend the original graph $G(V,E,f^l,f^u)$ so that all flow units required to fulfill the constraints on the bounds $f^l$ and $f^u$ have to use edge $(s',s)$ at a cost of $1$ each.
    The remaining flow can use the edge $(s',t')$ with a cost of $0$.
    An illustration of $G'(V',E',d,c,\ubar{f},\bar{f})$ can be found in Figure \ref{fig:mincost}.

    \begin{figure}[h!]
	\begin{center}
	   \resizebox{0.9\textwidth}{!}{
			\begin{tikzpicture}[->,>=stealth',shorten >=0.5pt,auto,node distance=1.7cm, thick,main node/.style={circle,draw,font=\sffamily\small}]
			\begin{scope}[every node/.style={circle,thick,draw}]
				\node (7) at (-3,0) {$ s' $};
				\node (8) at (9,0) {$ t' $};
                \node (0) at (1,0) {$ s $};
                \node (6) at (5,0) {$ t $};
		    \end{scope}
		    \begin{scope}[every node/.style={thick, black}]
			    \node (11) at (-3,-0.6) {\scriptsize $ d(s') =M$};
			    \node (22) at (9,-0.6) {\scriptsize $d(t') = -M$};
                \node (1) at (1.5,0.7) {};
			    \node (2) at (1.5,-0.7) {};
			    \node (3) at (2,0) {};
			    \node (4) at (4,0.7) {};
			    \node (5) at (4,-0.7) {};
		    \end{scope}
		    \path[every node/.style={font=\sffamily\footnotesize}]
    				(0) edge [bend left,dashed] node [ left ] 	{} 		(1)
    					edge [bend right,dashed] node [left] 	    {} 	(2)
                        edge [dashed]  node [below] 	    {} 	(3)
					(4) edge [bend left,dashed] node [ above ] 	{} 	(6)
					(5) edge [bend right,dashed] node [ below ] 	{} 	(6)
					(7) edge  node [ below ] {$c(s',s)=1$} node [ above ] {$(\ubar{f}_{s',s},\bar{f}_{s',s}) = (0,M)$}	(0)
					(6) edge  node [ below ] {$c(t,t')=0$} node [ above ] {$(\ubar{f}_{t,t'},\bar{f}_{t,t'}) = (0,M)$}	(8)
					(7) edge [bend left=50] node [ below ] {$c(s',t')=0$} node [ above ]{$(\ubar{f}_{s't'},\bar{f}_{s't'}) = (0,M)$} 	(8);

            \draw [-,decorate,decoration={brace,amplitude=5pt,mirror,raise=4ex}]
            (1,-0.4) -- (5,-0.4) node[midway,yshift=-4em]{$G(V,E,f^l,f^u)$};
		\end{tikzpicture}}
        \caption{Illustration of transformed graph $G'(V',E',d,c,\ubar{f},\bar{f})$ for the MCFP to solve the MWIFDP for $a_y = 0 $ and $a_w > 0$ in $ G(V,E,f^l,f^u) $.}
	   \label{fig:mincost}
        \end{center}
\end{figure}
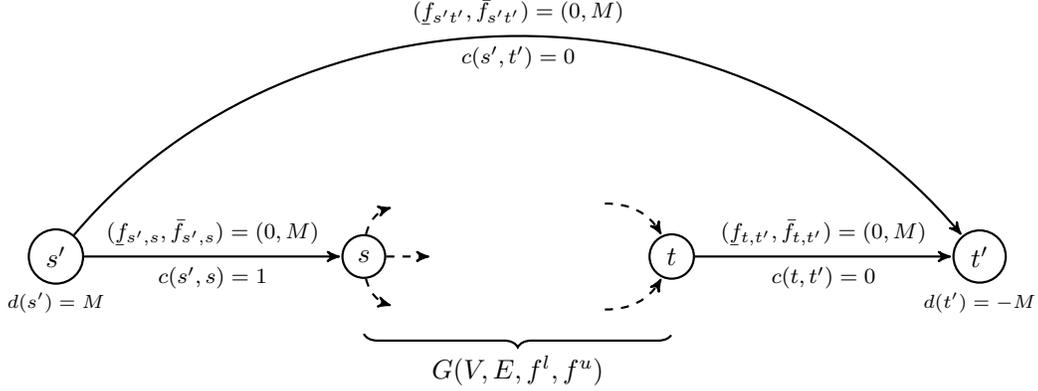
    
    If we now consider a minimum cost flow $f^m$ in $G'(V',E',d,c,\ubar{f},\bar{f})$, we can construct a feasible flow $f$ for the MWIFDP in the original graph $G(V,E,f^l,f^u)$ by removing edges $(s',s), (t,t'), (s',t')$ with corresponding flow values. 
    That means we have a flow $f \in \mathbb{N}_{\geq0}^{|E|}$ with $ f_{uv} := f^m_{uv} $ for $(u,v)\in E $ that is a feasible solution for MWIFDP by fulfilling the same flow and bound constraints.
    Since each flow unit in $f$ traverses edge $(s',s)$ with cost $1$, the total flow value equals the flow cost $c(f^m) := \sum_{(u,v) \in E'} c(u,v) \cdot f^m_{uv} = f^m_{s's}$ and is therefore minimal.
    Finally, notice that the number of $s$-$t$-paths can be arbitrarily large without affecting the objective function. In turn, we can trivially decompose $f$ into $c(f^m)$ $s$-$t$-paths each with weight $1$, which minimizes $a_w  \sum_{i=1}^{\overline{K}} w_i = a_w c(f^m)$ in \eqref{eq:MWIFD}.
\end{proof}

\subsection{Formulation without upper bounds}
For the remainder of this section, we neglect the upper bounds on the flow, i.e., we consider \eqref{eq:MWIFD} without constraints \eqref{eq:MWIFD5}.
This allows us to develop additional complexity results that are useful for both the MWIFDP and the strictly robust counterparts of the different variants to be presented in the next section.

\begin{lemma}\label{lemma:3}
    For $a_y = 0 $ and $a_w > 0$, the MWIFDP without upper bounds can be solved in polynomial time.
\end{lemma}
\begin{proof}
We follow the proof of Lemma \ref{lemma:2}, but since we have no upper bounds, we set them to $\infty $ for the MCFP, i.e., we have $(\ubar{f}_{uv},\bar{f}_{uv}) = (f^l_{uv},\infty)$. Consequently, the same reasoning also applies here: Since the flow constraints have to be satisfied and the lower bound for each $(u,v) \in E $ is $\ubar{f}_{uv}$, each edge (with $ \ubar{f}_{uv}> 0 $) is covered by one or more paths with a summed weight of at least $\ubar{f}_{uv}$.
Since each $s$-$t$-path is a $s'$-$t'$ path that uses edge $(s',s)$ with $c(s,s')=1$, the summed weights of all $s$-$t$-paths are minimized.
\end{proof}

If we consider the alternative case with $a_y > 0$ and $a_w = 0$, i.e., the MIFDP without upper bounds, we obtain a different complexity result than in Section \ref{subsec:generalformulation} (Corollary \ref{cor:cor7}).
\begin{lemma}\label{lemma:4}
    For $a_y > 0$ and $a_w = 0$, the MWIFDP without upper bounds can be solved in polynomial time.
\end{lemma}
\begin{proof}
    Due to the fact that the weights $w_i$ can be arbitrarily large without affecting the objective function, we can choose
\begin{equation*}
    w_i := \sum_{(u,v) \in E} f^l_{uv}
\end{equation*}
for $ i = 1,..., \overline{K}$ and the problem reduces to finding a minimal number of $s$-$t$-paths covering all edges with $f^l_{uv} > 0 $.

Once again, we can use the proof of Lemma \ref{lemma:1} to show that this corresponds to finding a minimal cost flow. However, there is a modification that is required for the lower bounds $\ubar{f}$, as we only want to cover the edges that have a positive flow. 
Therefore, it is sufficient to set $1$ as the lower bound for these.
Consequently, we have $(\ubar{f}_{uv},\bar{f}_{uv}) := (1,\infty)$ for all $(u,v) \in E$ with $f^l_{uv} > 0 $ and $(\ubar{f}_{uv},\bar{f}_{uv}) := (0,\infty)$ for $ (u,v) \in \{ (s',s), (t,t'), (s',t') \} \cup \{ (u,v) \in E: f^l_{uv} = 0 \}$.

Considering an optimal minimum-cost flow in this modified graph, 
each relevant edge is covered by at least one $s'$-$t'$-path (and thus $s$-$t$-path), while the first edge of each $s'-t'$ path has cost $1$.
Consequently, we minimize the number of $s$-$t$-paths covering all relevant edges.
\end{proof}

Lastly, we show that, although both previous cases can be solved in polynomial time, this is not the case for $a_y, a_w > 0$.

\begin{lemma}
    For $a_y, a_w > 0$, MWIFDP without upper bounds is NP-hard.
\end{lemma}
\begin{proof}
    We follow the proof of Proposition 2 in \cite{vatinlen2008simple}.
    Therefore, we show that the NP-hard 3-PARTITION problem \cite{garey1979computers} can be reduced to MWIFDP without upper bounds.
    
    An instance $(A,s)$ of 3-PARTITION consists of $3b$ Elements $a \in A$, where each element $a$ is associated with a size $s(a) \in \mathbb{N}_{\geq 1}$ with $\sum_{a \in A } s(a) = bB$ and $\frac{B}{4} < s(a) < \frac{B}{3}$.
    To decide whether there exists a partition of $A$ into $b$ disjoint subsets of size $B$ is known to be NP-hard.

    Now, we construct an instance $G(V,E,f^l)$ of MWIFDP without upper bounds associated with the instance $(A,s)$ of 3-PARTITION.
    We have $V = \{s,o,t \}$, $3b$ parallel edges $e'_a $ of the form $(s,o)$ corresponding to element $a \in A$ with $f^l_{e'_a} := s(a)$ and $b$ parallel edges $e''_i $ of the form $ (o,t)$ with $i \in \{1,...,b\}$ and $f^l_{e''_i} := B$.
    Evidently, we need at least $3b$ paths to cover the lower bounds of all parallel edges, i.e., 
    \begin{equation*}
        a_y |\mathcal{P}| \geq a_y \times 3b.
    \end{equation*}
    Since the sum of the lower bounds of the outgoing edges of $s$ is $\sum_{a \in A } f^l_{e'_a} = \sum_{a \in A } s(a) = bB$, the summed weights of the paths must therefore be at least $bB$, i.e.,
    \begin{equation*}
        a_w \sum_{i=1}^{|\mathcal{P}|} w_i \geq a_w \times bB.
    \end{equation*}
Following the same reasoning as in \cite{vatinlen2008simple}, there is a partition of $A$ into $b$ disjoint subsets with size $B$ if and only if there is a $3b$-flow decomposition into paths with a total weight of $bB$.
This means that the optimal objective value of MWIFDP without upper bounds for instance $G(V,E,f^l)$ is equal to
    \begin{equation*}
        a_y |\mathcal{P}| + a_w \sum_{i=1}^{|\mathcal{P}|} w_i = a_y \times 3b + a_w \times bB.
    \end{equation*}
 \end{proof}

\section{Robust minimum flow decomposition}\label{sec:robust_flows}

When considering uncertainty and (inexact) flow networks in general, it is natural that the flow values are the source of uncertainty. To provide formulations that can take into account the uncertainty associated with MFD flow values, we make use of two different notions of robustness, which, for the sake of completeness, we briefly describe next.

\subsection{Strict and adjustable robustness}\label{subsec:robustness}

Let $ \mathcal{U} \subset \mathbb{R}^N$ be an uncertainty set, i.e., a set of possible realizations for the uncertain parameter $ \xi \in \mathcal{U}$. For a given optimization under uncertainty problem of the form 
\begin{mini}
{x}{f(x)}
{\label{eq:uncdef}}{}
\addConstraint{F(x,\xi)}{\leq 0}
\addConstraint{x}{\in \mathbb{R}^n,}
\end{mini}    
with $f(\cdot): \mathbb{R}^n \rightarrow \mathbb{R} $ and $F(\cdot, \xi): \mathbb{R}^n \rightarrow \mathbb{R}^m $, the \emph{strictly robust counterpart} \cite{wei2018characterizations} is defined by
\begin{mini}
{x}{f(x)}
{\label{eq:count}}{}
\addConstraint{F(x,\xi)}{\leq 0}{ \forall\; \xi \in \mathcal{U}}
\addConstraint{x}{\in \mathbb{R}^n.}
\end{mini}    
An optimal solution to \eqref{eq:count} is called a \emph{strictly robust solution}, i.e., we search for the best solution that is feasible for all scenarios $\xi \in  \mathcal{U}$.

Strict robustness is often too conservative, especially regarding the requirement that the solution must be feasible for all scenarios within the uncertainty set.
As a result, the literature is rich in alternative forms of robustness, which have also been investigated \cite{beyer2007robust,gabrel2014recent}.

In some cases, it is reasonable to assume that part of the decisions, the so-called \emph{wait-and-see decisions}, can be made after the scenario $\xi$ is revealed, while the \emph{here-and-now decisions} have to be decided before the scenario $\xi$ is observed.
More formally, we can distinguish between \emph{non-adjustable variables} $u\in \mathbb{R}^{p}$ (here-and-now decisions, as in the strict robust case) and \emph{adjustable variables} $v\in \mathbb{R}^{q}$ (wait-and-see decisions), whose values may depend on the realization $\xi$.
This concept is called \emph{adjustable robustness} \cite{ben2004adjustable} and for $x=(u,v) \in \mathbb{R}^n $ the \emph{adjustable robust counterpart} of \eqref{eq:count} is defined as
\begin{mini}
{u \in S_u}{ \min_{v \in \mathbb{R}^{q}} f(u,v)}
{\label{eq:count2}}{}
\addConstraint{F(u,v,\xi)}{\leq 0\;}{ \forall\; \xi \in \mathcal{U}},
\end{mini}  
with $p+q=n$ and $S_u = \{u \in \mathbb{R}^{p} : \forall\; \xi \in \mathcal{U}, \, \exists v \in \mathbb{R}^{q} : F(u,v,\xi) \leq 0 \}$ representing general feasibility conditions for the choice of $u$.
This means that we make the decisions represented by the non-adjustable variables $u$, observe $\xi$, and then make decisions associated with the adjustable variables $v$ depending on the scenario $\xi$ and $u$.
An optimal solution to \eqref{eq:count2} is called an \emph{adjustable robust solution}. Note that for $p=n$, i.e., $x = u \in \mathbb{R}^n$, the adjustable robust counterpart \eqref{eq:count2} reduces to the strictly robust counterpart \eqref{eq:count}.

Using interval-wise uncertainty, we can define the scenarios according to the deviation from a nominal scenario $\hat{\xi}$ (e.g., the expected value or the most likely scenario).
This means that the uncertainty set is of the form 
\begin{equation}
    \cUint = \left\{ \xi \in \mathbb{R}_{\geq0}^{N} : \xi_i \in [\hat{\xi}_i - \ubar{\delta}_{i},\hat{\xi}_i +\bar{\delta}_{i}], \text{ } \forall\; i \in \{1,...,N \} \right\},\label{def:interval-wise}
\end{equation}
for given $\ubar{\delta},\bar{\delta} \in \mathbb{R}_{\geq0}^{N}$, where $\xi_i$ corresponds to the $i-$th entry of scenario $\xi$.

To prevent the uncertainty set from leading to over-conservative solutions, the cumulative deviation from the nominal scenario can be bounded by the uncertainty budget $ \Gamma \in \mathbb{R}_{\geq0} $.
More precisely, $\mathcal{U}(\Gamma)$ is of the form 
\begin{equation}
    \mathcal{U}(\Gamma) = \left\{ \xi \in \mathbb{R}_{\geq0}^{N} : \xi_i \in [\hat{\xi}_i - \ubar{\delta}_{i},\hat{\xi}_i +\bar{\delta}_{i}] , \text{ } \forall\; i \in \{ 1,...,N \} ; \text{ } \sum_{j=1}^{N} |\xi_j - \hat{\xi}_j| \leq \Gamma  \right\} .\label{def:gamma}
\end{equation}
This is known as \emph{cardinality-constraint uncertainty set} or \emph{budgeted uncertainty set} \cite{bertsimas2004price}; hereinafter, we refer to it as $\Gamma$-\emph{uncertainty set}.

\subsection{Applying the notion of strict robustness to the MWIFDP}\label{sec:robustMFD}

If we apply the notion of strict robustness to the (exact) MFDP with uncertain flows, the strictly robust counterpart \eqref{eq:count} becomes infeasible if there are at least two different scenarios $ f' \neq f'' \in \mathcal{U} $ since there is at least one $(u,v)$ with $f_{uv}' \neq f_{uv}''$ and constraints \eqref{eq:MFD4} cannot be fulfilled.

In general, this is not the case for the MIFDP since the summed path weights of each edge $(u,v)$ must lie within a range $[f^l_{uv},f^u_{uv}]$ instead of matching exactly a flow value $f_{uv}$.
Consequently, it is possible that different scenarios 
$ (f^l, f^u)' \neq (f^l,f^u)'' \in \mathcal{U} $  form a feasible combination of constraints in \eqref{eq:MIFD}. 

\begin{remark}
Since we are dealing with non-negative integer flows, non-integer bounds on the flow values can also be reduced to integers by rounding up or down.
Consequently, it is sufficient to consider a discrete uncertainty set consisting of values in $\mathbb{N}_{\geq0}$ since any other uncertainty set for this problem can be reduced to its discrete version.
Therefore, we set
\begin{equation}\label{eq:Uadj}
    \mathcal{U} = \left\{ (f^{l,\xi},f^{u,\xi}) \in \left(\mathbb{N}_{\geq0}^{|E|} \right)^2 : \xi = 1,...,N \right\},
\end{equation}
i.e., we have $|\mathcal{U}|$ scenarios, each scenario $\xi$ consisting of a lower $f^{l,\xi}_e$ and an upper bound $f^{u,\xi}_e$ for all edges $e \in E$ (w.l.o.g., we assume $f^{l,\xi}_e \leq f^{u,\xi}_e $).
\end{remark}

Having the uncertainty set $\mathcal{U}$ described as \eqref{eq:Uadj} significantly enhances tractability in our setting since we can show that the uncertainty set can be equivalently replaced by a singleton containing the \emph{worst-case} realization $\tilde{f}$. This is stated in Lemma \ref{lem:one_worst_case}. 


\begin{lemma}\label{lem:one_worst_case}
    For the strictly robust counterpart of MWIFDP, the uncertainty set $\mathcal{U}$ in \eqref{eq:Uadj} can be reduced to $\tilde{\mathcal{U}}=\{\tilde{f}\}$, i.e., there exists one worst-case scenario $\tilde{f} = (f^{l,\tilde{\xi}},f^{u,\tilde{\xi}})$.
\end{lemma}

\begin{proof}
    Concerning the strictly robust counterpart of MWIFDP, only constraints \eqref{eq:MWIFD4}-\eqref{eq:MWIFD5} are affected by the uncertainty.
    More precisely, we have
        \begin{align}
    		&& \sum_{i \in \{1,...,\overline{K}\}} w_i x_{uvi}                   &\geq f^{l,\xi}_{uv} &&  \forall\; (u,v) \in E,  &&  \forall\; (f^{l,\xi},f^{u,\xi}) \in \mathcal{U}, \label{eq:rcl} \\
            && \sum_{i \in \{1,...,\overline{K}\}} w_i x_{uvi}                   &\leq f^{u,\xi}_{uv} &&  \forall\; (u,v) \in E, &&  \forall\; (f^{l,\xi},f^{u,\xi}) \in \mathcal{U} \label{eq:rclu}.
        \end{align}
        Every feasible solution to MWIFDP with $f^l_{uv} = \max_{\xi \in \{1,...,N\}} f^{l,\xi}_{uv} $ and $ f^u_{uv} =\min_{\xi \in \{1,...,N\}} f^{u,\xi}_{uv}$ for all $ (u,v) \in E $ is feasible for all other scenarios $(f^{l,\xi},f^{u,\xi}) \in \mathcal{U}$.
    Consequently, for $(u,v) \in E$ we can write
    \begin{align*}
    		&& \sum_{i \in \{1,...,\overline{K}\}} w_i x_{uvi}                   &\geq f^{l,\xi}_{uv}  &&  \forall\; (f^{l,\xi},f^{u,\xi}) \in \mathcal{U}, \\
            \iff && \sum_{i \in \{1,...,\overline{K}\}} w_i x_{uvi}                   &\geq \max_{\xi \in \{1,...,N\}} f^{l,\xi}_{uv}  =: f^{l,\tilde{\xi}}_{uv}&&   \\
        \end{align*}
        and
        \begin{align*}
    		&& \sum_{i \in \{1,...,\overline{K}\}} w_i x_{uvi}                   &\leq f^{u,\xi}_{uv}  &&  \forall\; (f^{l,\xi},f^{u,\xi}) \in \mathcal{U}, \\
            \iff && \sum_{i \in \{1,...,\overline{K}\}} w_i x_{uvi}                   &\leq \min_{\xi \in \{1,...,N\}} f^{u,\xi}_{uv}  =: f^{u,\tilde{\xi}}_{uv} .&&   \\
        \end{align*}
        This means that we can project the uncertainty set $\mathcal{U}$ to two axes per edge associated with the lower and upper bounds and determine the maximum and minimum, respectively.
        Consequently, only the worst-case scenario $\tilde{f} = (f^{l,\tilde{\xi}}_{uv},f^{u,\tilde{\xi}}_{uv})_{(u,v) \in E}$ can be considered instead of the entire uncertainty set $\mathcal{U}$.
    \end{proof}

\begin{remark}
    Note that $\tilde{f} \notin \mathcal{U}$ is possible. This is an artifact of the static (as opposed to adjustable) setting of our strict robustness, where robustness is consequently considered constraint-wise \cite{ben2009robust}.
\end{remark}


For a general discrete uncertainty set $\mathcal{U}$ as defined in \eqref{eq:Uadj}, it is not immediately clear how to identify a worst-case scenario $\tilde{f}$. For an exponentially large uncertainty set, computing a worst-case scenario can be of exponential complexity. However, for $\mathcal{U}^\square$ in \eqref{def:interval-wise} and for $\mathcal{U}(\Gamma)$ in \eqref{def:gamma}, $\tilde{f}$ can be identified in linear time. Thus, we show based on the results from Section~\ref{sec:model} that strictly robust solutions can be computed in polynomial time.



First, consider an interval-wise uncertainty set
\begin{equation}\label{eq:unc}
     \cUint= \left\{ (f^l,f^u) \in \left(\mathbb{N}_{\geq0}^{|E|} \right)^2 : f^l_{uv} \in [\ubar{f}^l_{uv},\bar{f}^l_{uv}], f^u_{uv} \in [\ubar{f}^u_{uv},\bar{f}^u_{uv}], \text{ } \forall\; (u,v)\in E \right\},
\end{equation}
with given $\ubar{f}^l,\bar{f}^l,\ubar{f}^u,\bar{f}^u \in \mathbb{N}_{\geq0}^{|E|}$.
Note that since the flow values are integers, \eqref{eq:Uadj} contains \eqref{eq:unc} and can be constructed by enumerating all elements in an interval-wise uncertainty set.

\begin{lemma}\label{lemma:robustU}
    For the strictly robust counterpart of MWIFDP, the uncertainty set $\cUint$ 
    as defined in \eqref{eq:unc} can be reduced to $\mathcal{U} = \{ (f^l,f^u) \} $ with $f^l_{uv} =\bar{f}^l_{uv}$ and $ f^u_{uv} =\ubar{f}^u_{uv} $ for all $ (u,v) \in E $.
\end{lemma}
\begin{proof}
    It follows directly from the proof of \Cref{lem:one_worst_case} that
    every optimal solution to MWIFDP with $f^l_{uv} =\bar{f}^l_{uv}$ (maximum lower bound) and $ f^u_{uv} =\ubar{f}^u_{uv} $ (minimum upper bound) for all $ (u,v) \in E $ is a strictly robust solution.
\end{proof}


Now, if we restrict $\cUint$
in \eqref{eq:unc} to a $\Gamma$-uncertainty set $\mathcal{U}(\Gamma)$, i.e., $\mathcal{U}(\Gamma)$ must fulfill
\begin{equation}
    \sum_{(u,v)\in E} |f^l_{uv} - \hat{f}^l_{uv}| + |f^u_{uv} - \hat{f}^u_{uv}| \leq \Gamma
\end{equation}
concerning a nominal scenario $(\hat{f}^l,\hat{f}^u)$ with $\hat{f}^l_{uv} \in [\ubar{f}^l_{uv},\bar{f}^l_{uv}]$ and $ \hat{f}^u_{uv} \in [\ubar{f}^u_{uv},\bar{f}^u_{uv}]$ for $(u,v) \in E$, the resulting worst-case scenario is only slightly different.
\begin{lemma}\label{lemma:robustUG}
    For the strictly robust counterpart of MWIFDP, the $\Gamma$-uncertainty set $\mathcal{U}(\Gamma)$ can be reduced to $\mathcal{U}(\Gamma) = \{ (f^l,f^u) \} $ with $ f^l_{uv} = \min \{ \hat{f}^l_{uv} + \Gamma, \bar{f}^l_{uv} \} $ and $ f^u_{uv} = \max \{ \hat{f}^u_{uv} - \Gamma, \ubar{f}^u_{uv} \} $ for all $ (u,v) \in E $.
\end{lemma}
\begin{proof}
    We use the proof of \Cref{lem:one_worst_case} again, but this time we have to take the limited budget $\Gamma$ into account.
    Since the robust solution has to be feasible for any realization, a scenario that covers all worst-case scenarios is $(f^l_{uv},f^u_{uv}) = (\min \{ \hat{f}^l_{uv} + \Gamma, \bar{f}^l_{uv} \},\max \{ \hat{f}^u_{uv} - \Gamma, \ubar{f}^u_{uv} \})$ for all $(u,v) \in E$.
    Note that this particular scenario is not necessarily part of $\mathcal{U}(\Gamma)$.
\end{proof}


\begin{remark}
Other alternative geometries for the uncertainty set $\mathcal{U}$ 
can be considered in this setting. Following Lemma \ref{lem:one_worst_case}, the tractability of the strict robust versions of MFDP and MIFDP hinges on being able to identify the worst-case realization $\tilde{f}$ efficiently. 
\end{remark}

Since we have shown in \Cref{lem:one_worst_case} that any discrete uncertainty set of the form \eqref{eq:Uadj} can be reduced to a single worst-case scenario, the results from Section \ref{sec:model} also apply to the strictly robust counterpart of MWIFDP.

\begin{corollary}\label{lemma:2}
    For $a_y = 0 $ and $a_w > 0$, the strictly robust counterpart of MWIFDP with uncertainty sets $\cUint$ 
    and $\mathcal{U}(\Gamma)$ can be solved in polynomial time.
\end{corollary}
\begin{proof}
Using the same reasoning as before, we can reduce the uncertain optimization problem of MWIFDP to solving MWIFDP for the respective worst-case scenario for $\cUint$ 
and $\mathcal{U}(\Gamma)$.
Consequently, we use the procedure in Lemma \ref{lemma:1} with 
$(\ubar{f}_{uv},\bar{f}_{uv}) = (\bar{f}^l_{uv},\ubar{f}^u_{uv})$ and $(\ubar{f}_{uv},\bar{f}_{uv}) = (\min \{ \hat{f}^l_{uv} + \Gamma, \bar{f}^l_{uv} \},\max \{ \hat{f}^u_{uv} - \Gamma, \ubar{f}^u_{uv} \})$ for all $(u,v) \in E$ regarding the uncertainty sets $\cUint$ 
and $\mathcal{U}(\Gamma)$, respectively.
\end{proof}

If we neglect the upper bounds on the flow, the uncertainty sets $\cUint$ 
and $\mathcal{U}(\Gamma)$ simplify further.
Now, let $\cUint$ 
be an interval-wise uncertainty set with
\begin{equation}\label{eq:U}
     \cUint= \left\{ f^l \in \mathbb{N}_{\geq0}^{|E|} : f^l_{uv} \in [\ubar{f}^l_{uv},\bar{f}^l_{uv}], \text{ } \forall\; (u,v)\in E \right\},
\end{equation}
for given $\ubar{f}^l,\bar{f}^l \in \mathbb{N}_{\geq0}^{|E|}$ and $\mathcal{U}(\Gamma)$ the corresponding $\Gamma$-uncertainty set with
\begin{equation}\label{eq:UGamma}
    \mathcal{U}(\Gamma) = \left\{  f^l \in \mathbb{N}_{\geq0}^{|E|} : f^l_{uv} \in [\ubar{f}^l_{uv},\bar{f}^l_{uv}] , \text{ } \forall\; (u,v)\in E ; \text{ } \sum_{(u,v)\in E} |f^l_{uv} - \ubar{f}^l_{uv}| \leq \Gamma  \right\} .
\end{equation}
for a given uncertainty budget $\Gamma \in \mathbb{R}_{\geq0}$ and nominal case $\hat{f} = \ubar{f}^l$.
We can assume w.l.o.g. that $\Gamma \in \mathbb{N}_{\geq0}$ since $f^l \in \mathbb{N}_{\geq0}^{|E|}$.

\begin{remark}
Notice that the uncertainty sets in \eqref{eq:U} and \eqref{eq:UGamma} are discrete sets.
\end{remark}

Now, we can adopt the reasoning in Lemma \ref{lemma:robustU} and Lemma \ref{lemma:robustUG} for the lower bounds.
Consequently, a worst-case scenario in the strictly robust case is $\bar{f}^l_{uv}$ and $\min \{ \hat{f}^l_{uv} + \Gamma, \bar{f}^l_{uv} \}$ for all $(u,v) \in E $ regarding $\cUint$ 
and $\mathcal{U}(\Gamma)$, respectively. Moreover, one can notice that Corollary \ref{lemma:2} also holds for this modified version of MWIFDP without upper bounds.

\begin{corollary}
    For $a_y = 0 $ and $a_w > 0$, the strictly robust counterpart of MWIFDP without upper bounds and with uncertainty sets $\cUint$ 
    in \eqref{eq:U} and $\mathcal{U}(\Gamma)$ in \eqref{eq:UGamma} can be solved in polynomial time.
\end{corollary}
\begin{proof}
We follow the proof of Lemma \ref{lemma:3}.
That means we have $(\ubar{f}_{uv},\bar{f}_{uv}) = (\bar{f}^l_{uv},\infty)$ concerning uncertainty set $\cUint$ 
and $(\ubar{f}_{uv},\bar{f}_{uv}) = (\min \{ \hat{f}^l_{uv} + \Gamma, \bar{f}^l_{uv} \},\infty)$ concerning $\mathcal{U}(\Gamma)$.
\end{proof}

\begin{corollary}
    For $a_y > 0$ and $a_w = 0$, the strictly robust counterpart of MWIFDP without upper bounds and with uncertainty sets $\cUint$ 
    in \eqref{eq:U} and $\mathcal{U}(\Gamma)$ in \eqref{eq:UGamma} can be solved in polynomial time.
\end{corollary}
\begin{proof}
With the result from \Cref{lem:one_worst_case}, the uncertainty set $\cUint$ 
can be reduced to $\{ (f^l,f^u) \} $ with $f^l_{uv} =\bar{f}^l_{uv}$ and $ f^u_{uv} =\ubar{f}^u_{uv} $ for all $ (u,v) \in E $.
The same applies to the uncertainty set $\mathcal{U}(\Gamma)$, which can be reduced to $\{ (f^l,f^u) \} $ with $ f^l_{uv} = \min \{ \hat{f}^l_{uv} + \Gamma, \bar{f}^l_{uv} \} $ and $ f^u_{uv} = \max \{ \hat{f}^u_{uv} - \Gamma, \ubar{f}^u_{uv} \} $ for all $ (u,v) \in E $. Subsequently, we can follow the proof of Lemma \ref{lemma:4}. 
\end{proof}


\subsection{Adjustable robust flow decomposition}\label{sec:adj}

As we have seen in the previous section, even if we know that the uncertainty set is restricted and only a few edges in each scenario have a positive flow, we still need to set weights that cover all edges that have a positive flow in any scenario.
As a result, the number of paths (and the sum of weights) in the solution is much larger than what is actually required for any realized scenario.
Moreover, the uncertainty set must allow for a solution that is feasible for all scenarios.
To overcome the rigidity of the solutions obtained by the strictly robust model, we make use of adjustable robustness, introduced in Section \ref{subsec:robustness}, to allow part of the decisions to be made after the uncertainty has been observed.

Depending on the application at hand, it may be essential to know the number of paths a priori or to specify a superset of paths from which to choose.
In the context of transcript assembly, detaching the number of paths $y$ and each path($x$) from the weights ($w$) allows for accommodating errors in the measurement of the abundances, which in turn allows for more robust characterization of sequences despite the uncertainty on the actual splice graph.

Similarly, considering the field of transportation, we can have a situation where we need to know beforehand how many vehicles (corresponding to the number of paths $y$) will be required but can decide on capacities ($w$) and routes ($x$) later, as well as a situation where both the routes ($x$) and the number of drivers ($y$) must be known in advance, but the vehicles with suitable capacities ($w$) can be selected depending on the final realization.

As we can make three different types of decisions (number of paths, paths, weights) corresponding to the three different variable types ($y,x,w$), this naturally results in two adjustable robust optimization versions of MWIFDP based on \eqref{eq:MWIFD}.
In the first formulation, we have one here-and-now and two wait-and-see decisions, i.e., we first make a decision on the number of paths $y$ and then decide on the paths $x$ and weights $w$ depending on the observed scenario $\xi$.
Since we only have one set of non-adjustable variables, we call it the more adjustable formulation \MA.
Consequently, in the second less adjustable formulation, denoted by \LA, we first make decisions on the routes ($x$) and how many of them ($y$), and then we decide on the weights ($w$) depending on the scenario $\xi$.
An overview of both adjustable formulations can be found in Table \ref{tab:overviewadj}.

\begin{table}[h]
\centering
\resizebox{\textwidth}{!}{
 \begin{tabular}{p{0.16\textwidth}p{0.4\textwidth}p{0.4\textwidth}} 
 \toprule
 &  \MA{} (more adjustable formulation)  &  \LA{} (less adjustable formulation)  \\
 \midrule
               a priori & number of paths & number of paths, paths  \\
               a posteriori & paths, weights & weights \\
               non-adjustable& $y$ & $y$, $x$   \\
               adjustable & $x$, $w$ & $w$  \\
               \midrule
                solution method & Algorithm~\ref{alg:1V} & Algorithm~\ref{alg:2V} \\
                models & MIP formulation \eqref{eq:1V}, \newline master problem \eqref{eq:1VM}, \newline sub-problem \eqref{eq:1VS} & MIP formulation \eqref{eq:2V}, \newline master problem \eqref{eq:2VM}, \newline sub-problem \eqref{eq:2VS} \\
 \midrule
               features & $\bullet$ results in a lower number of paths  
               \newline $\bullet$ depending on the scenario, different paths might be used
                        & $\bullet$ results in a smaller pool of paths covering all scenarios \\
 \bottomrule
 \end{tabular}
 }
\caption{Overview of both adjustable formulations.}
\label{tab:overviewadj}
\end{table}



For \MA, the resulting MIP formulation is stated in \eqref{eq:1V}.
The adjustable variables $x^{\xi}$ and $w^{\xi}$ now depend on the scenario $\xi \in \mathcal{U}$ (given by the bounds $f^{l,\xi}$ and $f^{u,\xi}$) and, as a consequence, the sum of weights are moved to the constraints \eqref{eq:1V9} and substituted by variable $\mathcal{W}$ in the objective function.
\begin{mini!}
{x,y,w}{\sum_{i=1}^{\overline{K}} y_i + \mathcal{W} \label{eq:1V0}}
{\label{eq:1V}}{}
\addConstraint{\sum_{i=1}^{\overline{K}} w^{\xi}_i }{\leq \mathcal{W}, \label{eq:1V9}}{ \forall\; \xi \in \mathcal{U}}
\addConstraint{\sum_{v:(s,v) \in E} x^{\xi}_{svi}}{=y_i, \label{eq:1V1}}{ \forall\; i \in \{1,...,\overline{K}\}, \xi \in \mathcal{U}}
\addConstraint{\sum_{u:(u,t) \in E} x^{\xi}_{uti}}{=y_i, \label{eq:1V2}}{ \forall\; i \in \{1,...,\overline{K}\}, \xi \in \mathcal{U}}
\addConstraint{\sum_{(u,v) \in E} x^{\xi}_{uvi} - \sum_{(v,w) \in E} x^{\xi}_{vwi}}{=0, \label{eq:1V3}}{ \forall\; i \in \{1,...,\overline{K}\}, v \in V \setminus \{s,t\}, \xi \in \mathcal{U}}
\addConstraint{\sum_{i \in \{1,...,\overline{K}\}} w^{\xi}_i x^{\xi}_{uvi}}{\geq f^{l,\xi}_{uv}, \label{eq:1V4}}{\forall\; (u,v) \in E, \xi \in \mathcal{U}}
\addConstraint{\sum_{i \in \{1,...,\overline{K}\}} w^{\xi}_i x^{\xi}_{uvi}}{\leq f^{u,\xi}_{uv}, \label{eq:1V5}}{\forall\; (u,v) \in E, \xi \in \mathcal{U}}
\addConstraint{\mathcal{W}}{\in \mathbb{Z}^+  \label{eq:1V10}}
\addConstraint{w^{\xi}_i}{\in \mathbb{Z}^+,  \label{eq:1V6}}{ \forall\; i \in \{1,...,\overline{K}\}, \xi \in \mathcal{U}}
\addConstraint{x^{\xi}_{uvi}}{\in \{0,1\}, \quad \label{eq:1V7}}{ \forall\; (u,v) \in E,  i \in \{1,...,\overline{K}\}, \xi \in \mathcal{U}}
\addConstraint{y_i}{\in \{0,1\}, \label{eq:1V8}}{\forall\; i \in \{1,...,\overline{K}\}.}
\end{mini!}    

The formulation for \LA{} is similar, with the key difference being that only the variables $w^{\xi}$ are adjustable. The complete \LA{} formulation can be found in \eqref{eq:2V} in Appendix \ref{app:adj2}.

\begin{lemma}\label{lemma:}
    The objective value of the optimal solution to \MA{} in \eqref{eq:1V} is less than or equal to the objective value of the optimal solution to \LA{} in \eqref{eq:2V}.
\end{lemma}

\begin{proof}
    We show that all feasible solutions to \LA{} in \eqref{eq:2V} are also feasible to \MA{} in \eqref{eq:1V} and have the same objective value.
    Let $ (\Ddot{y},\Ddot{x},\Ddot{w})$ be a feasible solution to \LA{} with
    \begin{equation*}
        (\Ddot{y},\Ddot{x},\Ddot{w}) = (\Ddot{y}_i,\Ddot{x}_{uvi},\Ddot{w}^{\xi}_i)_{(u,v) \in E,  i \in \{1,...,\overline{K}\}, \xi \in \mathcal{U}}.
    \end{equation*}
    Since both formulations \eqref{eq:2V} and \eqref{eq:1V} differ only in the variable $x$ (\eqref{eq:2V7} and \eqref{eq:1V7}), we construct a solution $ (\dot{y},\dot{x},\dot{w})$ with
    \begin{equation*}
        (\dot{y},\dot{x},\dot{w}) = (\dot{y}_i,\dot{x}^{\xi}_{uvi},\dot{w}^{\xi}_i)_{(u,v) \in E,  i \in \{1,...,\overline{K}\}, \xi \in \mathcal{U}}
    \end{equation*}
    to \MA{} as follows.
    We set
        \begin{align*}
            && \dot{y}_i &:= \Ddot{y}_i, && \forall\; i \in \{1,...,\overline{K}\}, \\
            && \dot{w}^{\xi}_i &:= \Ddot{w}^{\xi}_i, && \forall\; i \in \{1,...,\overline{K}\}, \xi \in \mathcal{U}, \\
            && \dot{x}^{\xi}_{uvi} &:= \Ddot{x}_{uvi}, && \forall\; (u,v) \in E,  i \in \{1,...,\overline{K}\}, \xi \in \mathcal{U},
        \end{align*}
        whereby the latter is feasible because $\Ddot{x}$ satisfies constraints \eqref{eq:2V1}-\eqref{eq:2V5} for all scenarios $\xi \in \mathcal{U}$, and thus, $\dot{x}$ also fulfills constraints \eqref{eq:1V1}-\eqref{eq:1V5}.
        It is easy to see that the objective value has not changed, which completes the proof. Notice that this does not apply vice versa, as here the variables $x$ depend on scenario $\xi \in \mathcal{U}$ and are, therefore, generally not feasible for all scenarios.
\end{proof}

\begin{remark}
Naturally, \eqref{eq:1V} and \eqref{eq:2V} reduce to the MWIFDP formulation in \eqref{eq:MWIFD} if we have a single scenario, but unlike 
for strict robustness,
we need to consider more than one scenario in the general case.
Moreover, in contrast to strict robustness, an edge $e \in E$ for which the intersection of the intervals $ [f^{l,\xi}_e, f^{u,\xi}_e] $ of all scenarios $\xi$ is empty no longer automatically implies infeasibility.
\end{remark}


The MIP formulations \eqref{eq:1V} and \eqref{eq:2V} yield challenging problems with steep, impractical computational requirements. Therefore, to solve the proposed adjustable variants \LA{} and \MA, we adapt the \emph{column-and-constraint generation} approach in \cite{zeng2013solving} to find a minimal subset $\overline{\mathcal{U}}$ of $\mathcal{U}$ that is sufficient to \emph{cover} the adjustable problem for the complete uncertainty set $\mathcal{U}$.
Here, we say that $\overline{\mathcal{U}} \subset \mathcal{U} $ \emph{covers} $\mathcal{U}$ if solving the problem for $\overline{\mathcal{U}}$ returns the same solution as solving the problem for $\mathcal{U}$.
The adaptation relates to the fact that, while in \cite{zeng2013solving} the uncertainty set $\mathcal{U}$ is assumed to be continuous (with a finite number of extreme points and rays), our uncertainty set is discrete and finite.
Nevertheless, the convergence result in \cite{zeng2013solving} still guarantees that the proposed column-and-constraint generation algorithm will converge and return an optimal solution for a set $\overline{\mathcal{U}}$ that covers $\mathcal{U}$.

Therefore, we start by solving formulation \eqref{eq:1V} (or \eqref{eq:2V}) with subset $\overline{\mathcal{U}}$ that includes initially one arbitrary scenario.
This is the so-called \emph{master problem}.
We use the resulting solution of the non-adjustable variables to fix them in the \emph{sub-problem}, where we then solve the formulation for all scenarios $\xi \in \mathcal{U}$.
The scenario $\xi^*$ with the highest solution value or that causes infeasibility is added to $\overline{\mathcal{U}}$ and the steps are repeated until the solution value of the master problem equals the solution value of the sub-problem (in particular, it is feasible for all scenarios).
For \MA, a detailed description of the column-and-constraint generation method can be found in Algorithm \ref{alg:1V} and the corresponding formulations of the master problem and sub-problem in \eqref{eq:1VM} and \eqref{eq:1VS} in Appendix \ref{app:adj1}.
If a given solution $(y^{k})$ for the master problem leads to an infeasible sub-problem, we know that we need at least one more path.
As a consequence, in line 12 of Algorithm \ref{alg:1V}, we add the corresponding constraint to the master problem \eqref{eq:1VM}.
If this is the case in \LA, we only know that the combination of $(y^{k})$ and $(x^{k})$ causes infeasibility of the sub-problem.
Therefore, we add the constraint 
\begin{equation*}
    \sum_{i: (y^{k})_i=0} y_i +\sum_{i: (x^{k})_i=0} x_i + \sum_{i: (y^{k})_i=1} (1-y_i )+ \sum_{i: (x^{k})_i=1} (1-x_i ) \geq 1
\end{equation*}
to the master problem \eqref{eq:2VM} to eliminate this combination.
This can be found in Appendix \ref{app:adj2} in Algorithm \ref{alg:2V} with the corresponding formulations of the master problem and sub-problem of \LA{} in \eqref{eq:2VM} and \eqref{eq:2VS}.

\begin{algorithm}[h]
\caption{Column-and-constraint generation method for \MA}\label{alg:1V}
\begin{algorithmic}[1]
\State \textbf{given} a scenario set $\mathcal{U}$ and threshold $\epsilon$\;
\State set $LB = - \infty$, $UB = \infty$, $k=1$ and initialize $\overline{\mathcal{U}}$ with one scenario from $\mathcal{U}$\;
\While{$UB-LB > \epsilon $}
    \State solve master problem \eqref{eq:1VM}, get optimal solution $(y^{k},x^{k},w^{k},\mathcal{W}^{k})$\;
    \State update lower bound $LB =   \sum_{i=1}^{\overline{K}} (y^{k})_i + \mathcal{W}^{k} $\;
    \State solve sub-problem \eqref{eq:1VS} with input $(y^{k})$\;
    \If{we obtain an optimal solution $(\bar{x}^{k},\bar{w}^{k},\bar{\mathcal{W}}^{k})$}
        \State let $\xi^k$ be the worst-case scenario causing $\bar{\mathcal{W}}^{k}$\;
        \State update upper bound $UB =   \sum_{i=1}^{\overline{K}} (y^{k})_i + \bar{\mathcal{W}}^{k} $\;
    \Else
        \State let $\xi^k$ be the first scenario causing infeasibility\;
        \State add $ \sum_{i=1}^{\overline{K}} y_i \geq \sum_{i=1}^{\overline{K}} (y^{k})_i +1 $ to constraints of master problem \eqref{eq:1VM}\;
    \EndIf  
    \State add $\xi^k$ to $\overline{\mathcal{U}}$\;
    \State $k=k+1$\;
\EndWhile    
\State\Return $(y^{k},\bar{x}^{k},\bar{w}^{k})$\;
\end{algorithmic} 
\end{algorithm}


Even though it may be necessary in theory to consider all scenarios in the master problem, a much smaller number is sufficient in most cases, as we see in the experimental evaluation in Section \ref{sec:study}.

To evaluate the aforementioned adjustable formulations, we add a third conservative method, the so-called \emph{naive approach}.
Here, we solve the problem for each scenario separately and then add all paths with corresponding weights to a set of paths.
If solutions for different scenarios use the same path, we add the path only once.
To calculate $\mathcal{W}$, we use the maximum of the summed weights of the scenarios.
Consequently, this is a heuristic solution for \LA, giving us an upper bound on \LA{} and, thus, also on \MA.

\section{Computational study}\label{sec:study}

In this section, we experimentally evaluate both adjustable robust optimization versions of the MWIFDP introduced in the previous Section \ref{sec:adj} to quantify the benefit of adaptability when taking into account the uncertainty.
Moreover, we show that both \MA{} and \LA{} considerably outperform the naive approach, in which we solve the problem for each scenario separately.
As described in the previous section, we use the column-and-constraint generation methods in Algorithm \ref{alg:1V} and Algorithm \ref{alg:2V} as solution methods. The code and data generated are available at www.github.com/gamma-opt/robust-flow-decomposition.

\subsection{Data and scenarios}

Due to this being the first work to consider the MWIFDP and additionally adjustable robustness, we could not find instances in the literature that we could use directly for our purposes.
Therefore, we adapted graph-structured instances from the literature that are related to our work both theoretically and practically.
As mentioned in Section \ref{sec:introduction}, the most recent references concerning MFD stem from the field of bioinformatics, providing many datasets for the application of RNA transcript assembly.
Here, we select two of the instances also used in \cite{dias2022efficient} and \cite{dias2023accurate}, both being acyclic splice graphs, one for a human gene and one for a mouse gene.
Additionally, we added a benchmark instance for a minimum-cost flow problem from \cite{veldhorst1993bibliography} denoted by \emph{gte} and a graph that contains regional railway data of Lower Saxony, a region in northern Germany.
The latter comes from the scientific software toolbox LinTim \cite{schiewe2023documentation}, which deals with solving various planning steps in public transport.
Since this is not originally an $s$-$t$-flow network, we modified it by adding $71$ auxiliary edges for all missing outgoing edges of $s$ and incoming edges of $t$.
To complement this, we also used the artificial graph from Example \ref{ex:mfd} in Example \ref{ex:mfd}, which was also used in \cite{dias2023accurate}.
A description of all instances can be found in Table \ref{tab:instances} and the corresponding plots in Figure \ref{fig:examplea} and Figure \ref{fig:graphs}.

\begin{table}[h]
\centering
\scalebox{1}{
 \begin{tabular}{llrrrr}
 \toprule
 id &   name &  \#nodes &  \#edges &  origin &  plot \\
 \midrule
               1 & small & 7 &  12 &   Example \ref{ex:mfd} & Figure \ref{fig:examplea}\\
               2 & lowersaxony & 36 & 103 &  \cite{schiewe2023documentation} & Figure \ref{fig:lowersaxony} \\
               3 & human & 24 &39 & \cite{dias2022efficient,dias2023accurate} & Figure \ref{fig:human}  \\
               4 & mouse &111 & 120 &  \cite{dias2022efficient,dias2023accurate} & Figure \ref{fig:mouse}  \\
               5 & gte &  49 & 130 & \cite{veldhorst1993bibliography} & Figure \ref{fig:gte}  \\
 \bottomrule
 \end{tabular}
 }
\caption{Graph properties of the original instances.}
\label{tab:instances}
\end{table}


\begin{figure}[h!]
	\begin{center}
        \begin{subfigure}{0.49\linewidth}
	   \resizebox{1\textwidth}{!}{
        \includegraphics[width=1\linewidth]{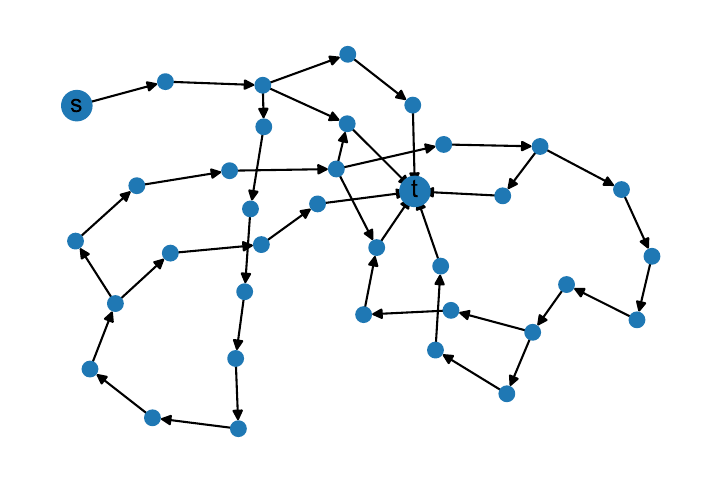}
        }
        \caption{lowersaxony}
        \label{fig:lowersaxony}
\end{subfigure}
\begin{subfigure}{0.49\linewidth}
	\resizebox{1\textwidth}{!}{
\includegraphics[width=1\linewidth]{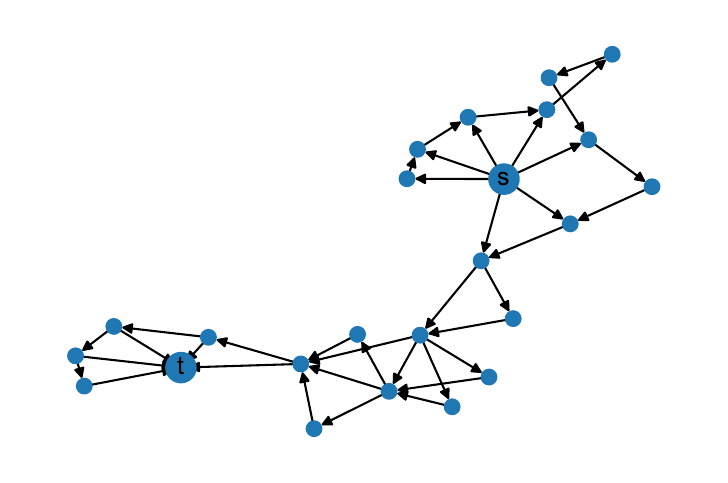}
        }
        \caption{human}
        \label{fig:human}
		\end{subfigure}
\begin{subfigure}{0.49\linewidth}
	\resizebox{1\textwidth}{!}{
\includegraphics[width=1\linewidth]{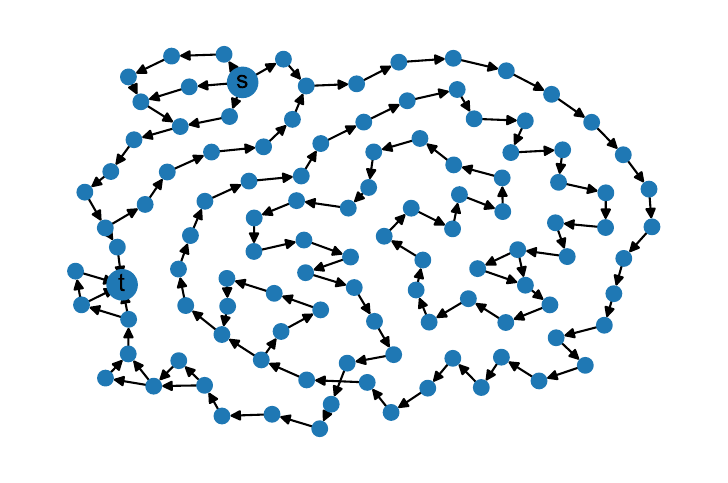}
        }
        \caption{mouse}
        \label{fig:mouse}
		\end{subfigure}
\begin{subfigure}{0.49\linewidth}
	\resizebox{1\textwidth}{!}{
\includegraphics[width=1\linewidth]{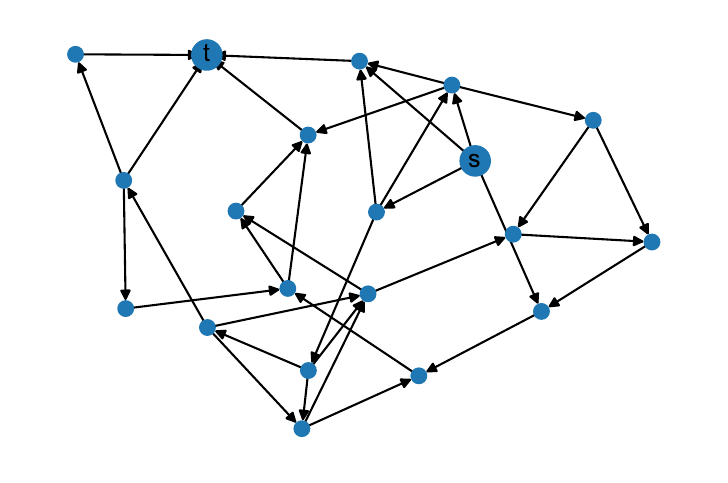}
        }
        \caption{gte}
        \label{fig:gte}
		\end{subfigure}
\end{center}
	\caption{Graph structures of the original instances.}
	\label{fig:graphs}
	\end{figure}

Deriving meaningful scenarios for the aforementioned instances is challenging. Besides the fact that each scenario must be feasible on its own, the difficulty lies in ensuring that the uncertainty set $\mathcal{U}(\Gamma)$ is sufficiently diverse to provide required protection levels whilst not being so disparate that they become unrealistic or impractical, rendering any robust consideration senseless.

To generate the scenarios that form the uncertainty set, we sample a random subset $\mathcal{P'}$ of $s$-$t$-paths of size $p\coloneqq|\mathcal{P'}|$ in the graph and set the flow value of each edge $e \in E$ to the number of occurrences of $e$ in these $s$-$t$-paths.
Consequently, $\mathcal{P'}$, with associated path weights all equal to $1$, forms a $|\mathcal{P'}|$-flow decomposition for the constructed flow network with weight $|\mathcal{P'}|$.
That means we guarantee feasibility and control the ratio of summed weights to be the number of paths since both are highly correlated and bounded by the subset size $|\mathcal{P'}|$.
We then define lower and upper bounds around these flow values for each edge $e \in E$ to construct one scenario.
Since we sample a new subset $\mathcal{P'}$ of $s$-$t$-paths for each scenario, this promotes variety in the scenarios forming the uncertainty set.

In order to measure and eliminate highly divergent scenarios from the uncertainty set, we first construct a nominal scenario ($f^*_{uv}, f^{l,*}_{uv},f^{u,*}_{uv}$, $(u,v)\in E$) using this method and then only include scenarios whose summed deviations lie within a certain range $\Gamma$ of the nominal scenario.
The detailed procedure is outlined in Algorithm \ref{alg:scen}.
Since in our computations, subset sizes $|\mathcal{P'}|$ of value $5, 10$, and $20$ produce very similar results, we use $|\mathcal{P'}| = 10$ in what follows.

\begin{algorithm}[h]  
\caption{Description of scenario generation}\label{alg:scen}
\begin{algorithmic}[1]
\State \textbf{given} digraph $G(V,E)$: 
\State compute all $s$-$t$-paths $\mathcal{P}(G)$ and set $\mathcal{U}(\Gamma) = \emptyset $, $\delta = \frac{p}{2}$\;
\State initialize $f_{uv}^* = 0 $ for all $ (u,v) \in E $
sample a random subset of paths $\mathcal{P'} \subset \mathcal{P}(G) $ with size $| \mathcal{P'} |=p$\;
\State update $f_{uv}^* = f_{uv}^* + 1 $ for all $ (u,v) \in P $ for each path $ P \in \mathcal{P'} $\;
\State set $f^{l,*}_{uv} = f_{uv}^* $ and $f^{u,*}_{uv} =f_{uv}^* + \delta $ for all $ (u,v) \in E $\;
$\xi = 1$
\While{$|\mathcal{U}(\Gamma)| <$ required number of scenarios}
    \State initialize $f_{uv} = 0 $ for all $ (u,v) \in E $\;
    \State sample a random subset of paths $\mathcal{P'} \subset \mathcal{P}(G) $ with size $| \mathcal{P'} |=p$\;
    \State update $f_{uv} = f_{uv} + 1 $ for all $ (u,v) \in P $ for each path $ P \in \mathcal{P'} $\;
    \State draw $f^{l,\xi}_{uv} \in [\max(0,f_{uv} - 2), f_{uv}] $ and set $f^{u,\xi}_{uv} =f_{uv} + \delta $ for all $ (u,v) \in E $\;
    \If{$ \sum_{e\in E} |f^{l,\xi}_{e} - f^{l,*}_{e}| + |f^{u,\xi}_{e} - f^{u,*}_{e}| \leq \Gamma' \cdot \sum_{e\in E} ( f^{l,*}_{e} + f^{u,*}_{e}) \eqqcolon \Gamma $}
        \State add scenario $(f^{l,\xi}, f^{u,\xi}) $ to uncertainty set $\mathcal{U}(\Gamma)$\;
    \EndIf    
    \State update $\xi = \xi + 1 $\;
\EndWhile

\State\Return $\mathcal{U}(\Gamma)$\;
\end{algorithmic}
\end{algorithm}

To define the range around the flow values for the lower and upper bounds
, we use $\delta = \frac{|\mathcal{P'}|}{2}$.
Note that $f_{uv}$ is bounded by $|\mathcal{P'}|$.
Due to the size of the graphs, $f_{uv}$ is closer to $0$ in most cases, which is why we draw the lower bound in the interval $[\max(0,f_{uv} - 2), f_{uv}]$ for all $ (u,v) \in E $.
Note that all auxiliary edges have a lower bound of zero, 
i.e., we set the lower bounds $f^{l,*}_{uv} $ and $f^{l,\xi}_{uv} $ for the \emph{lowersaxony} instance to $0$ if $u=s$ or $v=t$. 

\subsection{Experimental results}\label{subsec:results}

As mentioned in the previous section, we choose subset size $|\mathcal{P'}| = 10$. For $\Gamma'$, we use $0.1, 0.2$, and $0.3$, i.e., the allowed summed deviations to the nominal scenario lie within a range of $10$ \%, $20$ \%, and $30$ \%, respectively.
As for the uncertainty set sizes, we use $|\mathcal{U}(\Gamma)| \in \{5,10,50,100,200,500\}$ since the differences in the results become smaller for the larger sizes.
Due to the complexity of the problem, we set a time limit of $24$ hours for executing the complete algorithm while ensuring full iterations of the master problem and sub-problem.
This means that if the time limit is not reached after solving the master problem (or sub-problem), we do not interrupt the computation of the sub-problem (respectively, master problem), even if the time limit is exceeded in the meantime.

We set a time limit of $30$ minutes for the master problem.
If no feasible solution is obtained within this time, the algorithm execution continues until one is found. For the sub-problem, it is sufficient to limit the time to $3$ minutes. Here, we take advantage of the fact that it can be solved individually for each scenario instead of considering them all at once. This significantly reduces the computational burden, but it also means that the solution time of the sub-problem scales with the number of scenarios, as they are solved sequentially, although this could be overcome by parallelization. For the naive approach, we use a maximum of one hour per scenario.

\begin{table}[h]
\centering
\scalebox{1}{
 \begin{tabular}{lrrrrr}
 \toprule
 model &  small &  lowersaxony &  human &  mouse &   gte \\
 \midrule
               \MA       & 382.56    & 6200.08     &  3687.06   &  78.04    & 11289.26 \\
               \LA       & 52873.67  & 47353.45    & 60309.20   & 84434.52  & 66440.38 \\
               naive    & 788.61    & 1817.98    & 2481.08 & 831.86  & 3825.67 \\
 \bottomrule
 \end{tabular}
 }
\caption{Runtimes for the different formulations and instances in seconds, averaged over all uncertainty set sizes and values for $\Gamma'$.}
\label{tab:runtime}
\end{table}

As already indicated, the adjustable version of MWIFDP is computationally challenging, which is what motivates the employment of the column-and-constraint generation algorithm.
Indeed, the master problem alone, consisting of an adjustable MWIFDP (with a reduced number of scenarios that increases as iterations progress), is in itself computationally demanding. While in \MA, there is only one set of non-adjustable variables, and thus $x$ and $w$ can be chosen individually depending on the scenario, in \LA, both $y$ and $x$ must be feasible (and in the best case optimal) for all scenarios.
This leads to the less adjustable formulation \LA{} being significantly harder to solve, which is reflected in the runtimes and number of iterations reported in Table \ref{tab:runtime} and Table \ref{tab:iterations}, respectively.
The numbers provided refer to the average of all uncertainty set sizes $|\mathcal{U}(\Gamma)|$ and values for $\Gamma'$.
Note that we use a pre-solving phase to initialize the value for the lower bound, which we do not count in the number of iterations.
More precisely, after initializing $|\mathcal{U}(\Gamma)|$, we solve the master problem without upper bounds on the flow values and also without $\mathcal{W}$ in the objective function (corresponding to Lemma \ref{lemma:4}).
The resulting value for the number of paths is used to initialize $LB$.
\begin{table}[h]
\centering
\scalebox{1}{
 \begin{tabular}{lrrrrr}
 \toprule
 model &  small &  lowersaxony &  human &  mouse &   gte \\
 \midrule
               \MA       & 2.00    & 4.72     &  1.72   &  1.17    & 2.94 \\
               \LA       & 31.33  & 30.39    & 30.22   & 39.95  & 36.44 \\
 \bottomrule
 \end{tabular}
 }
\caption{Iterations for the different formulations and instances, averaged over all uncertainty set sizes and values for $\Gamma'$.}
\label{tab:iterations}
\end{table}

While \MA{} rarely reaches the time limit and usually requires only a few iterations to converge, the opposite is true for \LA.
The particularly low mean runtimes of \MA{} for \emph{small} and \emph{mouse} can be explained by the small graph size and simple structure of the graph (very long paths without branches), respectively.
For \LA, however, these are dominated by the additional challenge incurred by the less adjustable master problem.
Moreover, the fact that with \MA, we have two adjustable variable sets causes the sub-problem to no longer be solved in a few seconds, as with \LA, where only the optimal values for $w$ need to be computed.
However, since we can consider the sub-problem individually for each scenario, the computational requirements are still moderate.
Nevertheless, the fact that we do not solve the scenarios in parallel means that the runtime for \MA{} can be significantly longer for larger uncertainty set sizes.
As a consequence, some of the solutions showcased below are not necessarily optimal. 

\begin{figure}[h!]
	\begin{center}
        \begin{subfigure}{0.49\linewidth}
	   \resizebox{1\textwidth}{!}{
        \includegraphics[width=1\linewidth]{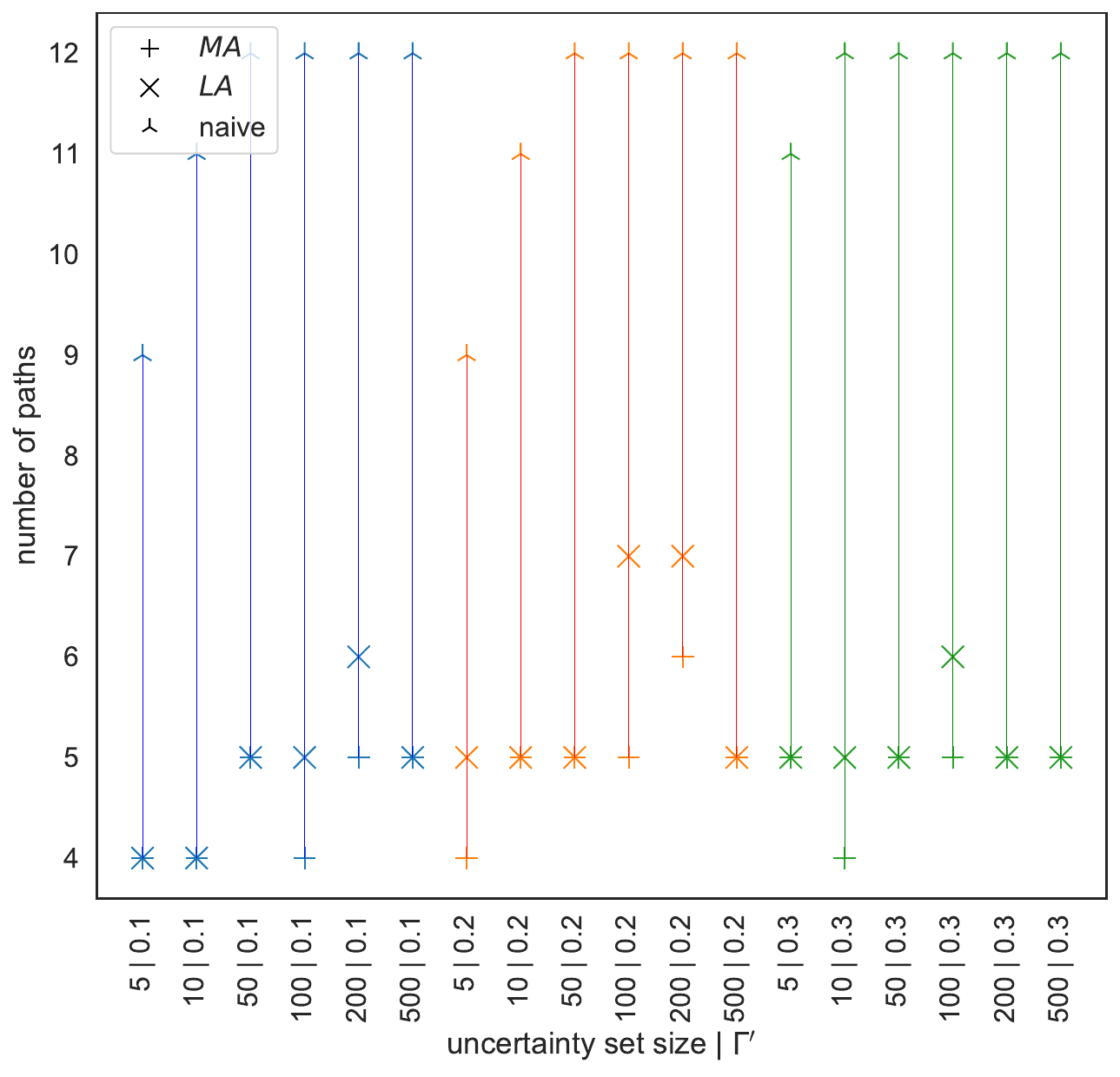}
        }
        \caption{Values of $\mathcal{Y} := \sum_i y_i$}
        \label{fig:res_small_y}
\end{subfigure}
\begin{subfigure}{0.49\linewidth}
	\resizebox{1\textwidth}{!}{
\includegraphics[width=1\linewidth]{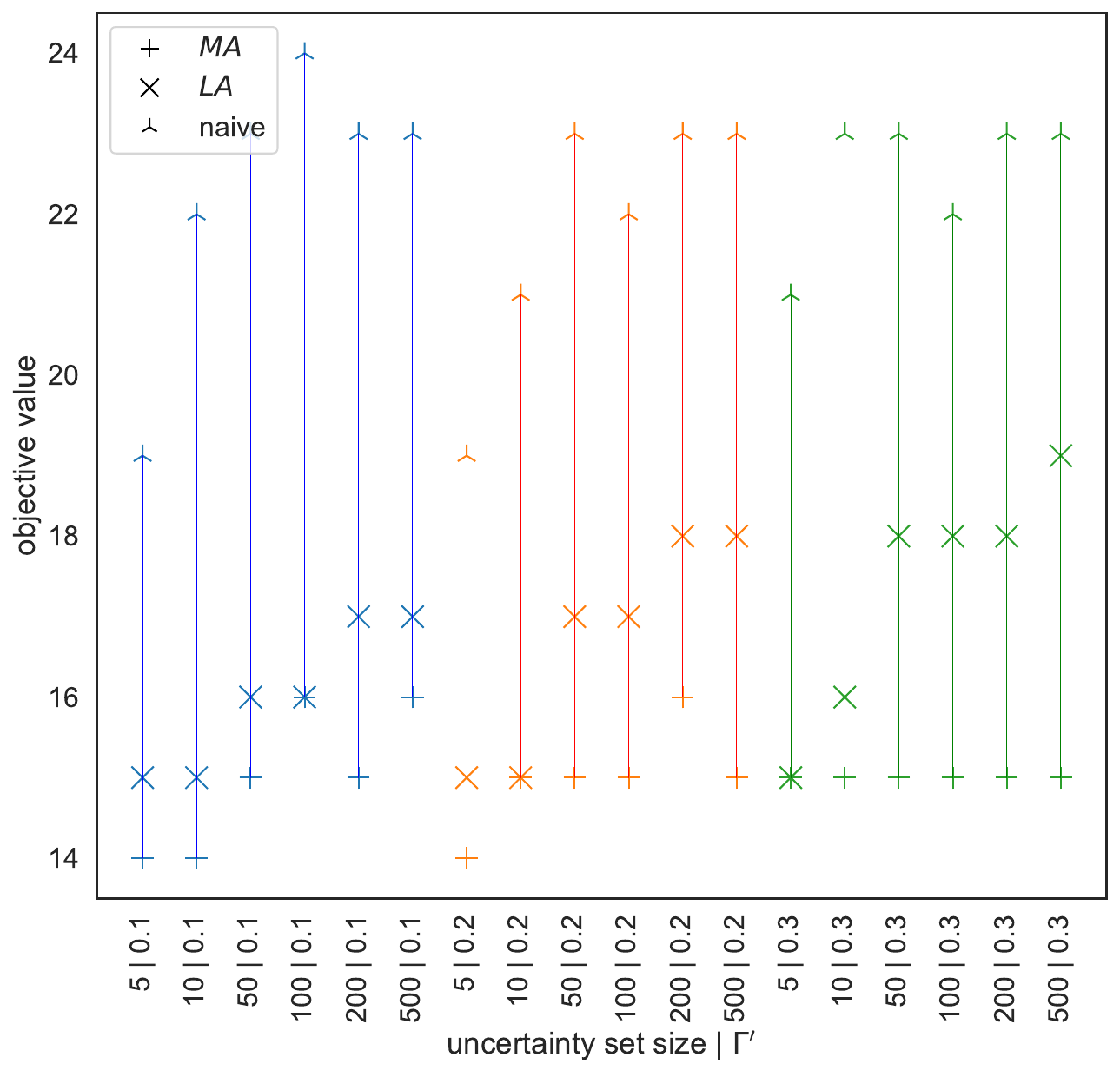}
        }
        \caption{Objective values}
        \label{fig:res_small_obj}
		\end{subfigure}
\end{center}
	\caption{Results of \MA, \LA{} and the naive approach for instance set \emph{small} for the different uncertainty set sizes $ |\mathcal{U}(\Gamma)|$ and values for $\Gamma'$.}
	\label{fig:res_small}
	\end{figure}	

For the sake of clarity, we only show detailed results for the instance sets \emph{small}, \emph{lowersaxony}, and \emph{human}.
All other results can be found in Appendix \ref{app:results}.
For \emph{small}, Figure \ref{fig:res_small} shows both the values for $\sum_i y_i := \mathcal{Y}$ (number of paths) and the objective values for the different uncertainty set sizes and values for $\Gamma$.
The objective values for \emph{lowersaxony} and \emph{human} are shown in Figures \ref{fig:res_sax} and \ref{fig:res_human}, respectively.
As we can see, the values for \MA{} are consistently the lowest and those of the naive approach the highest.
This also aligns with theoretical reasoning, which states that the less adjustable formulation \LA{} can never achieve better objective values than \MA{} and is bounded from above by the objective function value of the naive approach. 
These theoretical properties apply to the objective value but not necessarily to the number of paths.
Nevertheless, the experimental results also show the same behavior for the latter.

 \begin{figure}[h!]
	\begin{center}
        \begin{subfigure}{0.49\linewidth}
	   \resizebox{1\textwidth}{!}{
        \includegraphics[width=1\linewidth]{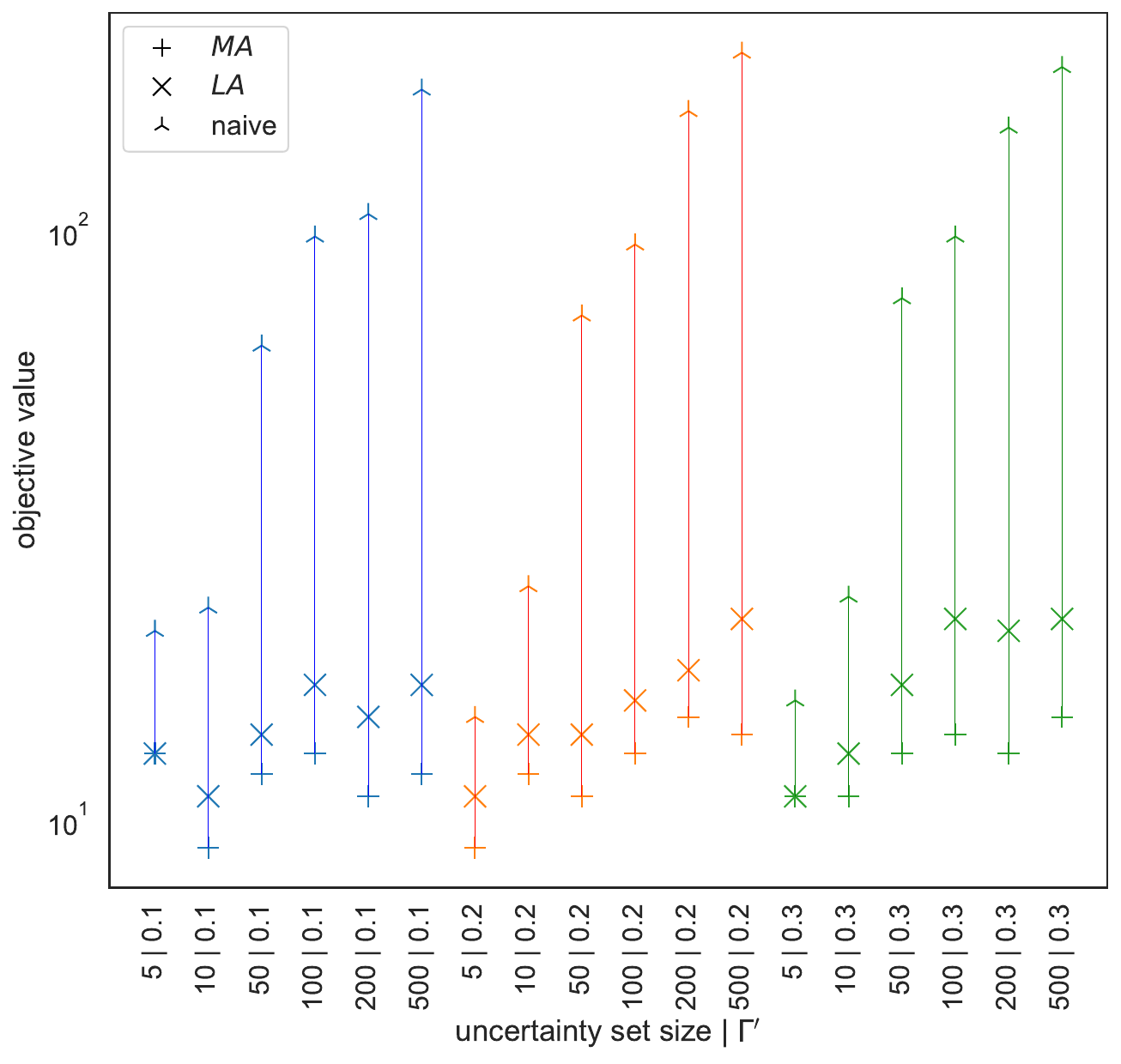}
        }
        \caption{Results for instance set \emph{lowersaxony}}
        \label{fig:res_sax}
\end{subfigure}
\begin{subfigure}{0.49\linewidth}
	\resizebox{1\textwidth}{!}{
\includegraphics[width=1\linewidth]{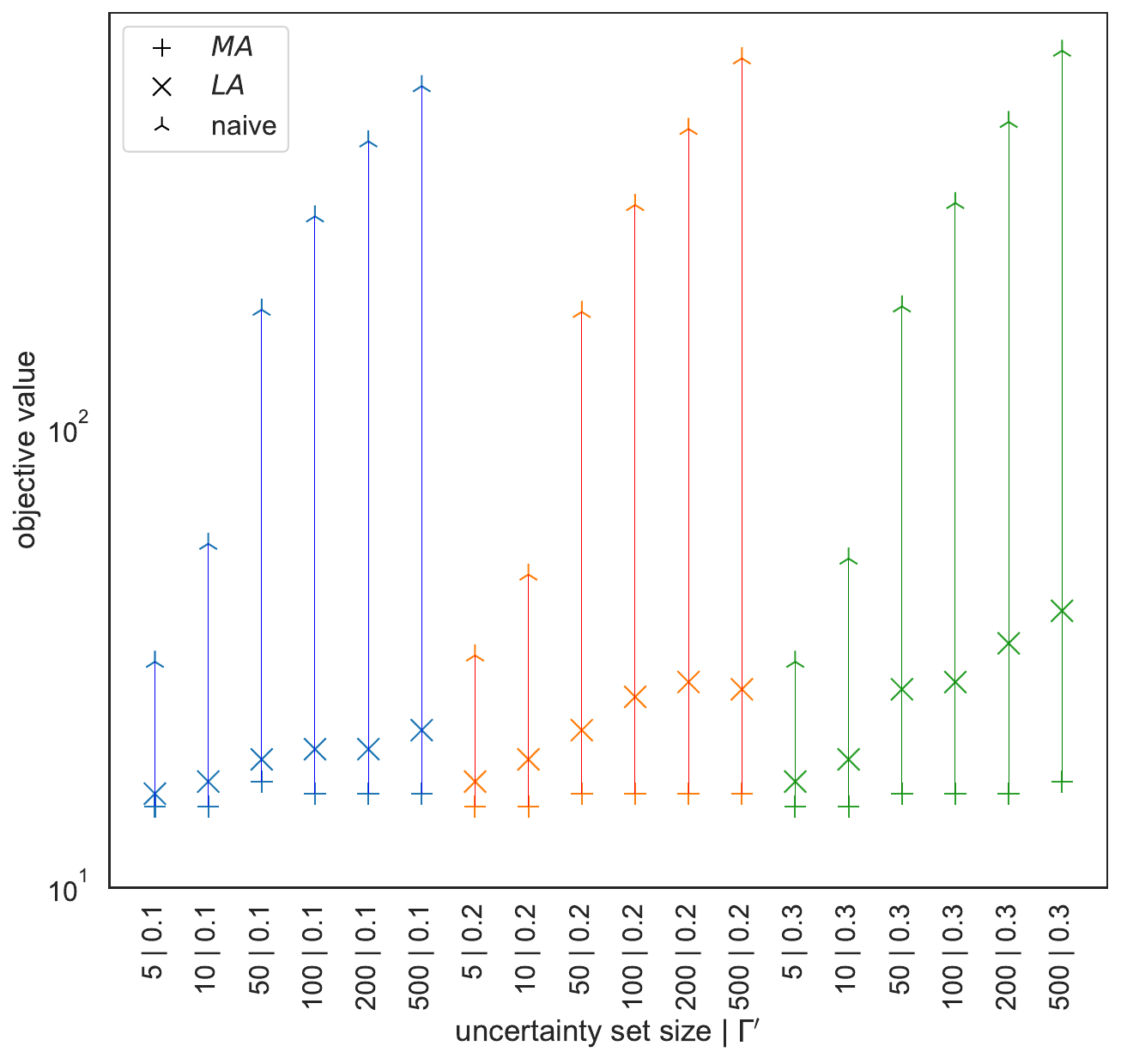}
        }
        \caption{Results for instance set \emph{human}}
        \label{fig:res_human}
		\end{subfigure}
\end{center}
	\caption{Objective values of \MA, \LA{} and the naive approach for the different uncertainty set sizes $ |\mathcal{U}(\Gamma)|$ and values for $\Gamma'$ (logarithmic scale).}
	\label{fig:res_obj}
	\end{figure}		

Since we use a set of $s$-$t$-paths of size $| \mathcal{P'} |$ to create a scenario, we know that we can decompose each scenario into a maximum of $| \mathcal{P'} |$ $s$-$t$-paths.
This means that in \MA, $\sum_i y_i$ is bounded by $| \mathcal{P'} |$ since we only have to determine the maximum number of paths that is sufficient per scenario.
For \LA, in addition to the number, the paths themselves must also be specified, forming a feasible solution for all scenarios. 
Consequently, the union of the original $s$-$t$-paths of all scenarios is feasible, i.e., $\sum_i y_i$ is bounded by $ |\mathcal{U}(\Gamma)| \cdot | \mathcal{P'} |$.
With this theoretical background, however, it can be seen from the results that these bounds are not even close to being reached. 
Depending on the instance, the maximum number of paths ranges between 5 and 8 for \MA{}, 7 and 19 for \LA, and 12 to 652 for the naive approach, while the theoretical upper bound is 5000. 

In terms of the summed weights $\mathcal{W}$, the results of \MA{} and the naive approach are similar (between $6$ and $12$) and hardly vary for different values of $ |\mathcal{U}(\Gamma)|$ and $ \Gamma'  $. 
This can be explained by the fact that with both methods, the paths and weights can be adjusted to each other and chosen to suit the scenario.
In contrast, \LA{} searches for a set of paths that fits all scenarios. 
Consequently, it becomes harder to match the paths with the different values for $w$ to the individual scenario, which leads to significantly higher values of $\mathcal{Y}$ and $\mathcal{W}$ (between $6$ and $21$) and also to both increasing with growing $ |\mathcal{U}(\Gamma)|$ and $ \Gamma' $.

When looking at $\mathcal{Y}$, we see a similar picture for \MA{} as we did before with $\mathcal{W}$.
Due to the more comprehensive adjustability, the values are relatively small (between $3$ and $8$), and the increase with growing $ |\mathcal{U}(\Gamma)|$ and $  \Gamma'$ is also modest.
The opposite is the case with the naive approach. 
Since the set of paths is generated by the union of the solutions for the individual scenarios, in order to be feasible for all scenarios (as in \LA), $\mathcal{Y}$ increases with $ |\mathcal{U}(\Gamma)|$ (up to a value of $652$ and significantly more than in \LA), which also leads to the objective value being dominated by $\mathcal{Y}$.
This is particularly noticeable with more complex (but not necessarily larger) graphs, where the number of different paths (and thus the objective value) is significantly higher and can be well observed in Figures \ref{fig:res_obj} and \ref{fig:res_naiveV1}.
In the latter, one can see how the ratio of the number of paths between the naive approach and \MA{} increases significantly for the more complex graphs and with larger values for $ |\mathcal{U}(\Gamma)|$ (depending on the instance, the naive approach requires up to $104$ times more paths than \MA).
This underlines the advantage of our adjustable models, which provide a substantial improvement compared to the naive approach.
Even the less adjustable version \LA{} leads to a significant reduction in the number of paths, which can be seen in Figure \ref{fig:res_V2V1}.
It shows the box plots for the five different instance classes.
These associated values are calculated from the difference in the number of paths of \LA{} and \MA{} divided by the difference in the number of paths of the naive approach and \MA.
In other words, it indicates as a percentage how close the $\mathcal{Y}$ value of version \LA{} is to the naive approach starting from \MA.
As can be seen in Figure~\ref{fig:res_V2V1}, almost all values are below $0.15$.
The only exception is instance \emph{gte}, for which the corresponding values range between $0.2$ and $0.4$, which is still a significant improvement compared to the naive approach.

 \begin{figure}[h!]
	\begin{center}
        \begin{subfigure}{0.475\linewidth}
	   \resizebox{1\textwidth}{!}{
        \includegraphics[width=1\linewidth]{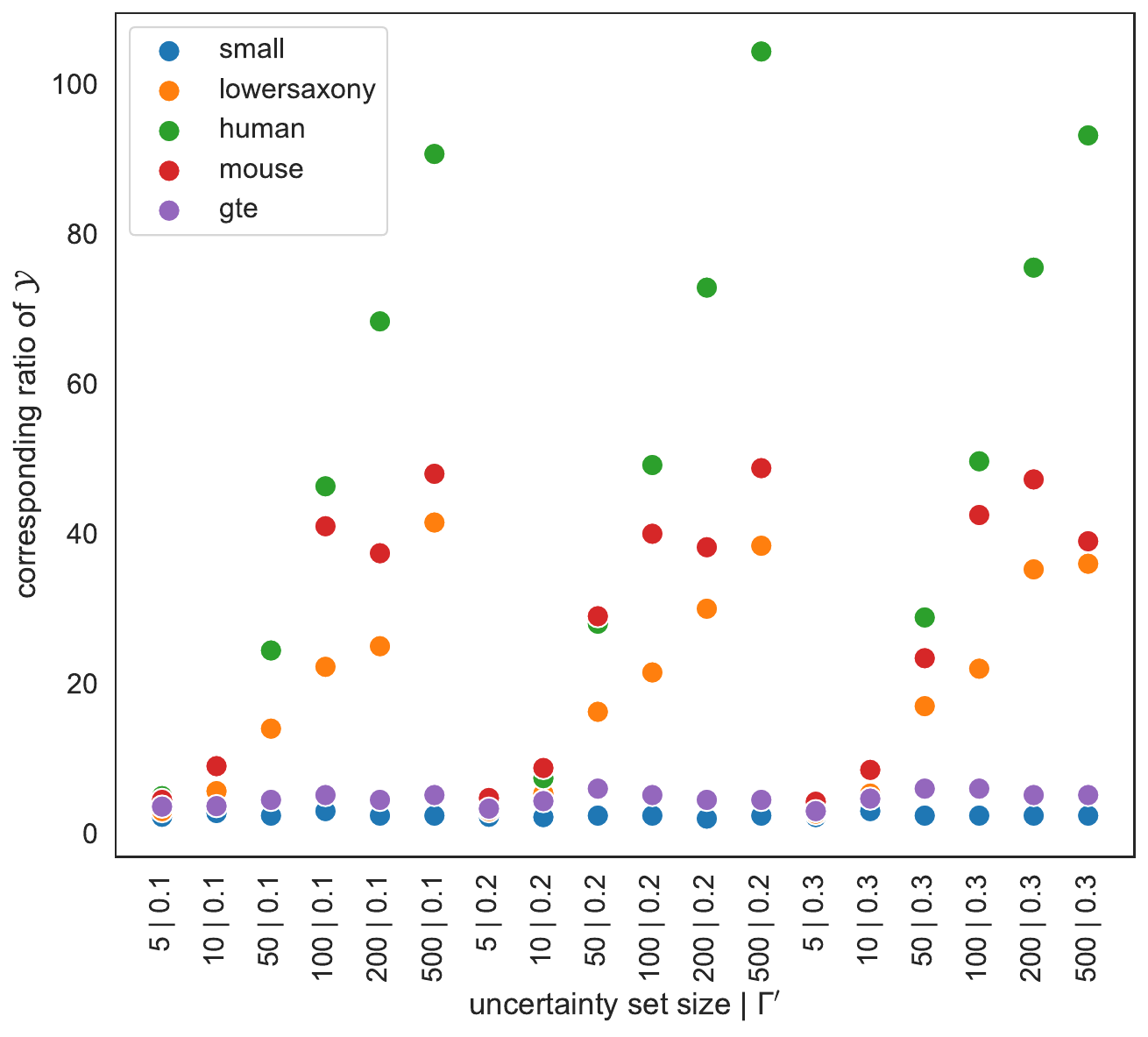}
        }
        \caption{$\mathcal{Y}$ of the naive approach divided by $\mathcal{Y}$ of \MA{} for the different uncertainty set sizes $ |\mathcal{U}(\Gamma)|$ and values for $\Gamma'$, i.e., $\frac{\mathcal{Y}^{\textup{naive}}}{\mathcal{Y}^{\MA}}$}
        \label{fig:res_naiveV1}
\end{subfigure}
\hfill 
\begin{subfigure}{0.475\linewidth}
	\resizebox{1\textwidth}{!}{
\includegraphics[width=1\linewidth]{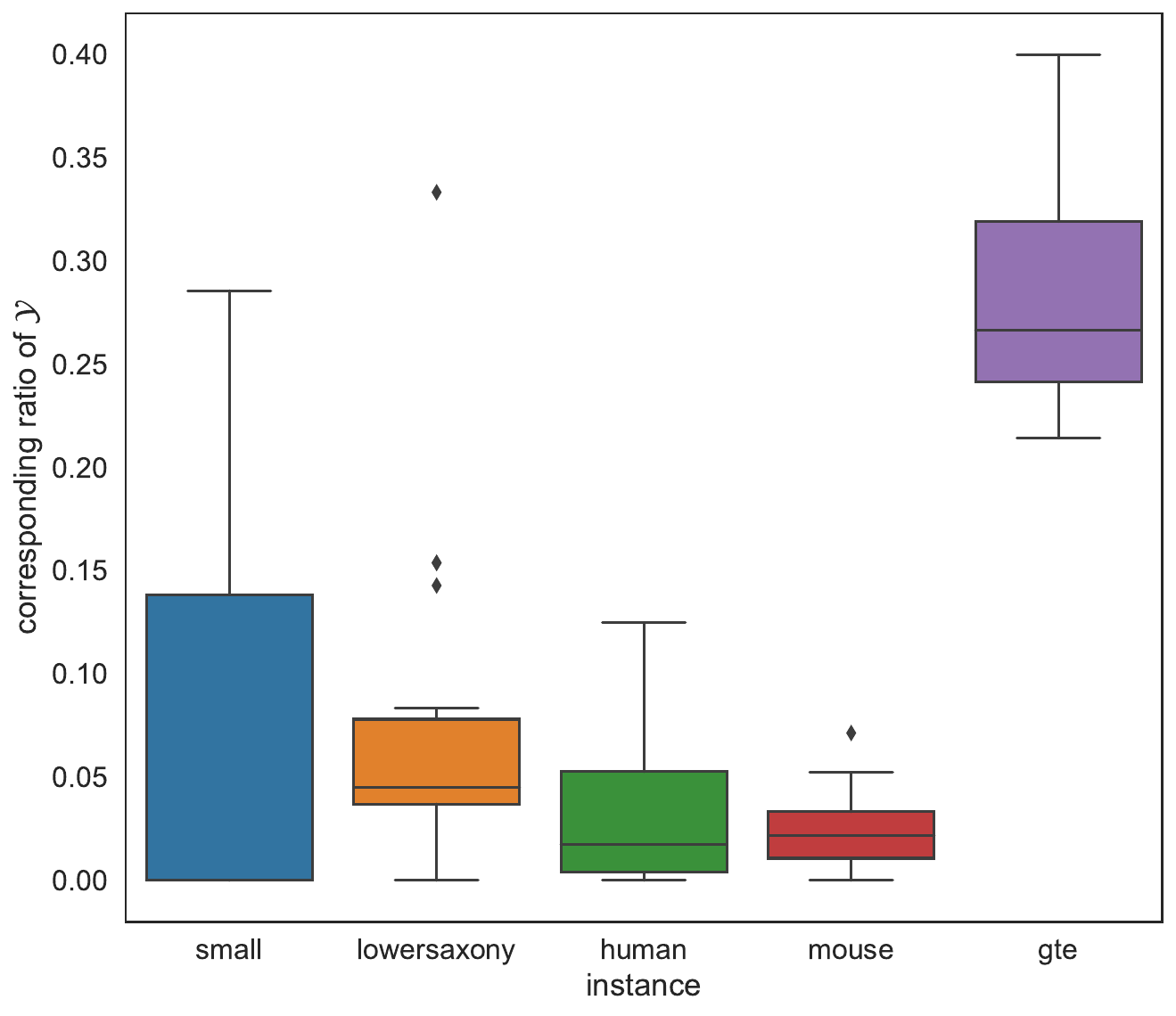}
        }
        \caption{Difference in $\mathcal{Y}$ of \LA{} and \MA{} divided by the difference in $\mathcal{Y}$ of the naive approach and \MA, i.e., $\frac{\mathcal{Y}^{\LA} -\mathcal{Y}^{\MA}}{\mathcal{Y}^{\textup{naive}}-\mathcal{Y}^{\MA}}$}
        \label{fig:res_V2V1}
		\end{subfigure}
\end{center}
	\caption{Values of $\mathcal{Y}:=\sum_i y_i$ (number of paths) of the three algorithms in relation to each other for the different instance sets.}
	\label{fig:res_y}
	\end{figure}

To summarize the results, it can be noted that the adjustability has a considerable advantage over the non-adjustable model, and this is clearly reflected in the number of paths.
Furthermore, even in the less adjustable \LA{} case, our models deliver significantly better results than the naive approach at the expense of dealing with more computationally challenging problems.

\section{Conclusions}\label{sec:concl}

In this paper, we address the minimum flow decomposition problem and consider it from the perspective of robust optimization. Unlike typical bioinformatics applications that focus on robustness against measurement errors, we explore classical robustness concepts from the literature. Our focus also lies on applications related to transportation planning problems, such as optimizing vehicle utilization in public transport.

Motivated by these applications, we generalize the problem to the weighted inexact case with lower and upper bounds on the flow values. We investigate the complexity of the resulting problem and show that we can solve the special case where we only minimize the weights in polynomial time. Moreover, we introduce the concept of robust flow decomposition by incorporating uncertain bounds. Here, we show that arbitrary uncertainty sets can be reduced to finite, discrete uncertainty sets and that for strict robustness, the uncertainty set reduces to a single worst-case scenario.



To overcome the conservatism of strict robustness, and 
motivated by practical applications, we make use of adjustable robustness.
Consequently, we present two different adjustable problem formulations.
In the first, only the number of paths needs to be specified beforehand. In the second formulation, both the number and the specific paths must be decided before the scenario is revealed.

Our computational study provides a proof of concept, demonstrating the benefits of adjustability in settings under uncertainty when compared to non-adjustable models, which is clearly reflected in the number of paths.
Even the less adjustable model yields significantly better results than the approach in which the problem is solved separately for each scenario, and all paths with corresponding weights are aggregated into a pool.
Our findings highlight the substantial advantages of our adjustable models.

However, the gained flexibility comes at the cost of higher computational requirements.
Thus, focusing on specialized solution methods that aim at improving runtimes would be a valuable contribution to future research. In particular, methods that can influence the choice of effective scenarios, both for initializing the algorithm and during its progression, based on external factors (such as data, as proposed in \cite{goerigk2023data}) could play a significant role in improving computational performance. Moreover, exploring alternative uncertainty set formulations may expose not only better trade-offs in terms of robustness guarantees but also favor computational performance.
Finally, since previous concepts have mainly focused on dealing with data inaccuracies, and this is the first work to consider MFDs in the context of classical robustness concepts, this opens up additional fields of application and offers opportunities for further research.

\section*{Authors' declarations}

\paragraph{Conflicts of interest/Competing interests.}The authors declare no conflicting or competing interests.

\paragraph{Availability of data and material.}The code and data generated are available at \linebreak www.github.com/gamma-opt/robust-flow-decomposition.

\paragraph{Authors' contributions.}M.S. is the main author, taking responsibility for writing the manuscript, preparing all visual materials, developing the technical analyses, and executing the computational experiments. P.S. contributed by revising the manuscript, providing essential technical support for the analysis, and contributing to delimiting the scope and designing the computational studies. F.O. contributed by revising the manuscript, delimiting the scope, and designing the computational studies. 

\paragraph{Ethics approval, Consent to participate, Consent for publication, and Funding}Not applicable.

\printbibliography

\appendix

\section{Adjustable formulations}\label{app:adj}
\subsection{Master problem and sub-problem for \MA}\label{app:adj1}

Formulations of the master problem \eqref{eq:1VM} and sub-problem \eqref{eq:1VS} for \MA{} in Algorithm \ref{alg:1V} in Section \ref{sec:adj}.

\begin{mini!}
{x,y,w}{\sum_{i=1}^{\overline{K}} y_i + \mathcal{W} \label{eq:1VM0}}
{\label{eq:1VM}}{}
\addConstraint{\sum_{i=1}^{\overline{K}} w^{\xi}_i }{\leq \mathcal{W}, \label{eq:1VM9}}{ \forall\; \xi \in \overline{\mathcal{U}}}
\addConstraint{\sum_{v:(s,v) \in E} x^{\xi}_{svi}}{=y_i, \label{eq:1VM1}}{ \forall\; i \in \{1,...,\overline{K}\}, \xi \in \overline{\mathcal{U}}}
\addConstraint{\sum_{u:(u,t) \in E} x^{\xi}_{uti}}{=y_i, \label{eq:1VM2}}{ \forall\; i \in \{1,...,\overline{K}\}, \xi \in \overline{\mathcal{U}}}
\addConstraint{\sum_{(u,v) \in E} x^{\xi}_{uvi} - \sum_{(v,w) \in E} x^{\xi}_{vwi}}{=0, \label{eq:1VM3}}{ \forall\; i \in \{1,...,\overline{K}\}, v \in V \setminus \{s,t\}, \xi \in \overline{\mathcal{U}}}
\addConstraint{\sum_{i \in \{1,...,\overline{K}\}} w^{\xi}_i x^{\xi}_{uvi}}{\geq f^{l,\xi}_{uv}, \label{eq:1VM4}}{\forall\; (u,v) \in E, \xi \in \overline{\mathcal{U}}}
\addConstraint{\sum_{i \in \{1,...,\overline{K}\}} w^{\xi}_i x^{\xi}_{uvi}}{\leq f^{u,\xi}_{uv}, \label{eq:1VM5}}{\forall\; (u,v) \in E, \xi \in \overline{\mathcal{U}}}
\addConstraint{\mathcal{W}}{\in \mathbb{Z}^+  \label{eq:1VM10}}
\addConstraint{w^{\xi}_i}{\in \mathbb{Z}^+,  \label{eq:1VM6}}{ \forall\; i \in \{1,...,\overline{K}\}, \xi \in \overline{\mathcal{U}}}
\addConstraint{x^{\xi}_{uvi}}{\in \{0,1\}, \quad \label{eq:1VM7}}{ \forall\; (u,v) \in E,  i \in \{1,...,\overline{K}\}, \xi \in \overline{\mathcal{U}}}
\addConstraint{y_i}{\in \{0,1\}, \label{eq:1VM8}}{\forall\; i \in \{1,...,\overline{K}\}.}
\end{mini!} 

\begin{mini!}
{x,y,w}{\sum_{i=1}^{\overline{K}} y_i + \mathcal{W} \label{eq:1VS0}}
{\label{eq:1VS}}{}
\addConstraint{\sum_{i=1}^{\overline{K}} w^{\xi}_i }{\leq \mathcal{W}, \label{eq:1VS9}}{ \forall\; \xi \in \mathcal{U}}
\addConstraint{\sum_{v:(s,v) \in E} x^{\xi}_{svi}}{=y_i, \label{eq:1VS1}}{ \forall\; i \in \{1,...,\overline{K}\}, \xi \in \mathcal{U}}
\addConstraint{\sum_{u:(u,t) \in E} x^{\xi}_{uti}}{=y_i, \label{eq:1VS2}}{ \forall\; i \in \{1,...,\overline{K}\}, \xi \in \mathcal{U}}
\addConstraint{\sum_{(u,v) \in E} x^{\xi}_{uvi} - \sum_{(v,w) \in E} x^{\xi}_{vwi}}{=0, \label{eq:1VS3}}{ \forall\; i \in \{1,...,\overline{K}\}, v \in V \setminus \{s,t\}, \xi \in \mathcal{U}}
\addConstraint{\sum_{i \in \{1,...,\overline{K}\}} w^{\xi}_i x^{\xi}_{uvi}}{\geq f^{l,\xi}_{uv}, \label{eq:1VS4}}{\forall\; (u,v) \in E, \xi \in \mathcal{U}}
\addConstraint{\sum_{i \in \{1,...,\overline{K}\}} w^{\xi}_i x^{\xi}_{uvi}}{\leq f^{u,\xi}_{uv}, \label{eq:1VS5}}{\forall\; (u,v) \in E, \xi \in \mathcal{U}}
\addConstraint{\mathcal{W}}{\in \mathbb{Z}^+  \label{eq:1VS10}}
\addConstraint{w^{\xi}_i}{\in \mathbb{Z}^+,  \label{eq:1VS6}}{ \forall\; i \in \{1,...,\overline{K}\}, \xi \in \mathcal{U}}
\addConstraint{x^{\xi}_{uvi}}{\in \{0,1\}, \quad \label{eq:1VS7}}{ \forall\; (u,v) \in E,  i \in \{1,...,\overline{K}\},  \xi \in \mathcal{U}.}
\end{mini!}    

\subsection{Adjustable formulation \LA}\label{app:adj2}

Adjustable MIP formulation \eqref{eq:2V} for \LA{} in Section \ref{sec:adj}, where the wait-and-see variables $w^{\xi}$ depend on the scenario $\xi \in \mathcal{U}$ (given by the bounds $f^{l,\xi}$ and $f^{u,\xi}$).

\begin{mini!}
{x,y,w}{\sum_{i=1}^{\overline{K}} y_i + \mathcal{W} \label{eq:2V0}}
{\label{eq:2V}}{}
\addConstraint{\sum_{i=1}^{\overline{K}} w^{\xi}_i }{\leq \mathcal{W}, \label{eq:2V9}}{ \forall\; \xi \in \mathcal{U}}
\addConstraint{\sum_{v:(s,v) \in E} x_{svi}}{=y_i, \label{eq:2V1}}{ \forall\; i \in \{1,...,\overline{K}\}}
\addConstraint{\sum_{u:(u,t) \in E} x_{uti}}{=y_i, \label{eq:2V2}}{ \forall\; i \in \{1,...,\overline{K}\}}
\addConstraint{\sum_{(u,v) \in E} x_{uvi} - \sum_{(v,w) \in E} x_{vwi}}{=0, \label{eq:2V3}}{ \forall\; i \in \{1,...,\overline{K}\}, v \in V \setminus \{s,t\}}
\addConstraint{\sum_{i \in \{1,...,\overline{K}\}} w^{\xi}_i x_{uvi}}{\geq f^{l,\xi}_{uv}, \label{eq:2V4}}{\forall\; (u,v) \in E, \xi \in \mathcal{U}}
\addConstraint{\sum_{i \in \{1,...,\overline{K}\}} w^{\xi}_i x_{uvi}}{\leq f^{u,\xi}_{uv}, \label{eq:2V5}}{\forall\; (u,v) \in E, \xi \in \mathcal{U}}
\addConstraint{\mathcal{W}}{\in \mathbb{Z}^+  \label{eq:2V10}}
\addConstraint{w^{\xi}_i}{\in \mathbb{Z}^+,  \label{eq:2V6}}{ \forall\; i \in \{1,...,\overline{K}\}, \xi \in \mathcal{U}}
\addConstraint{x_{uvi}}{\in \{0,1\}, \quad \label{eq:2V7}}{ \forall\; (u,v) \in E, i \in \{1,...,\overline{K}\} }
\addConstraint{y_i}{\in \{0,1\}, \label{eq:2V8}}{\forall\; i \in \{1,...,\overline{K}\}.}
\end{mini!}

The corresponding column-and-constraint generation approach (based on \cite{zeng2013solving}) to solve the adjustable problem \LA{} in \eqref{eq:2V} can be found in Algorithm \ref{alg:2V}.

As described in Section \ref{sec:adj}, if a given solution $(y^{k},x^{k})$ for the master problem leads to an infeasible sub-problem, we know that we can eliminate the combination of $(y^{k})$ and $(x^{k})$ from the solution space.
As a consequence, in line 13 of Algorithm \ref{alg:2V}, we add the corresponding constraint to the master problem \eqref{eq:2VM}.

\begin{algorithm}[h]
\caption{Column-and-constraint generation method for \LA} \label{alg:2V}
\begin{algorithmic}[1]
\State \textbf{given} a scenario set $\mathcal{U}$ and threshold $\epsilon$\;
\State Set $LB = - \infty$, $UB = \infty$, $k=1$ and initialize $\overline{\mathcal{U}}$ with one scenario from $\mathcal{U}$\;
\While{$UB-LB > \epsilon $}
    \State solve master problem \eqref{eq:2VM}, get optimal solution $(y^{k},x^{k},w^{k},\mathcal{W}^{k})$\;
    \State update lower bound $LB =   \sum_{i=1}^{\overline{K}} (y^{k})_i + \mathcal{W}^{k} $\;
    \State solve sub-problem \eqref{eq:2VS} with input $(y^{k},x^{k})$\;
    \If{we obtain an optimal solution $(\bar{w}^{k},\bar{\mathcal{W}}^{k})$}
        \State let $\xi^k$ be the worst-case scenario causing $\bar{\mathcal{W}}^{k}$\;
        \State update upper bound $UB =   \sum_{i=1}^{\overline{K}} (y^{k})_i + \bar{\mathcal{W}}^{k} $\;
    \Else
        \State let $\xi^k$ be the first scenario causing infeasibility\;
        \State add $\sum_{i: (y^{k})_i=0} y_i +\sum_{i: (x^{k})_i=0} x_i + \sum_{i: (y^{k})_i=1} (1-y_i )+ \sum_{i: (x^{k})_i=1} (1-x_i ) \geq 1$ to constraints of master problem \eqref{eq:2VM}\;
        \State add $\xi^k$ to $\overline{\mathcal{U}}$\;
    \EndIf    
\State $k=k+1$\;
\EndWhile
\State\Return $(y^{k},x^{k},\bar{w}^{k})$\;
\end{algorithmic}
\end{algorithm}

The described master problem and sub-problem of \LA{} in Algorithm \ref{alg:2V} are shown in \eqref{eq:2VM} and \eqref{eq:2VS}, respectively.

	\begin{mini!}
{x,y,w}{\sum_{i=1}^{\overline{K}} y_i + \mathcal{W} \label{eq:2VM0}}
{\label{eq:2VM}}{}
\addConstraint{\sum_{i=1}^{\overline{K}} w^{\xi}_i }{\leq \mathcal{W}, \label{eq:2VM9}}{ \forall\; \xi \in \overline{\mathcal{U}}}
\addConstraint{\sum_{v:(s,v) \in E} x_{svi}}{=y_i, \label{eq:2VM1}}{ \forall\; i \in \{1,...,\overline{K}\}}
\addConstraint{\sum_{u:(u,t) \in E} x_{uti}}{=y_i, \label{eq:2VM2}}{ \forall\; i \in \{1,...,\overline{K}\}}
\addConstraint{\sum_{(u,v) \in E} x_{uvi} - \sum_{(v,w) \in E} x_{vwi}}{=0, \label{eq:2VM3}}{ \forall\; i \in \{1,...,\overline{K}\}, v \in V \setminus \{s,t\}}
\addConstraint{\sum_{i \in \{1,...,\overline{K}\}} w^{\xi}_i x_{uvi}}{\geq f^{l,\xi}_{uv}, \label{eq:2VM4}}{\forall\; (u,v) \in E, \xi \in \overline{\mathcal{U}}}
\addConstraint{\sum_{i \in \{1,...,\overline{K}\}} w^{\xi}_i x_{uvi}}{\leq f^{u,\xi}_{uv}, \label{eq:2VM5}}{\forall\; (u,v) \in E, \xi \in \overline{\mathcal{U}}}
\addConstraint{\mathcal{W}}{\in \mathbb{Z}^+  \label{eq:2VM10}}
\addConstraint{w^{\xi}_i}{\in \mathbb{Z}^+,  \label{eq:2VM6}}{ \forall\; i \in \{1,...,\overline{K}\}, \xi \in \overline{\mathcal{U}}}
\addConstraint{x_{uvi}}{\in \{0,1\}, \quad \label{eq:2VM7}}{ \forall\; (u,v) \in E,  i \in \{1,...,\overline{K}\} }
\addConstraint{y_i}{\in \{0,1\}, \label{eq:2VM8}}{\forall\; i \in \{1,...,\overline{K}\}.}
\end{mini!}  

	\begin{mini!}
{x,y,w}{\sum_{i=1}^{\overline{K}} y_i + \mathcal{W} \label{eq:2VS0}}
{\label{eq:2VS}}{}
\addConstraint{\sum_{i=1}^{\overline{K}} w^{\xi}_i }{\leq \mathcal{W}, \label{eq:2VS9}}{ \forall\; \xi \in \mathcal{U}}
\addConstraint{\sum_{i \in \{1,...,\overline{K}\}} w^{\xi}_i x_{uvi}}{\geq f^{l,\xi}_{uv}, \label{eq:2VS4}}{\forall\; (u,v) \in E, \xi \in \mathcal{U}}
\addConstraint{\sum_{i \in \{1,...,\overline{K}\}} w^{\xi}_i x_{uvi}}{\leq f^{u,\xi}_{uv}, \label{eq:2VS5}}{\forall\; (u,v) \in E, \xi \in \mathcal{U}}
\addConstraint{\mathcal{W}}{\in \mathbb{Z}^+  \label{eq:2VS10}}
\addConstraint{w^{\xi}_i}{\in \mathbb{Z}^+,  \label{eq:2VS6}}{ \forall\; i \in \{1,...,\overline{K}\}, \xi \in \mathcal{U}.}
\end{mini!}

\section{Detailed Results}\label{app:results}

Complementing Section \ref{subsec:results}, the remaining results can be found in the following.
For \emph{mouse} and \emph{gte}, Figures \ref{fig:res_mouse} and \ref{fig:res_gte} show both the values for $\mathcal{Y}:=\sum_i y_i$ and the objective values for the different uncertainty set sizes and values for $\Gamma$.

\begin{figure}[h!]
	\begin{center}
        \begin{subfigure}{0.49\linewidth}
	   \resizebox{1\textwidth}{!}{
        \includegraphics[width=1\linewidth]{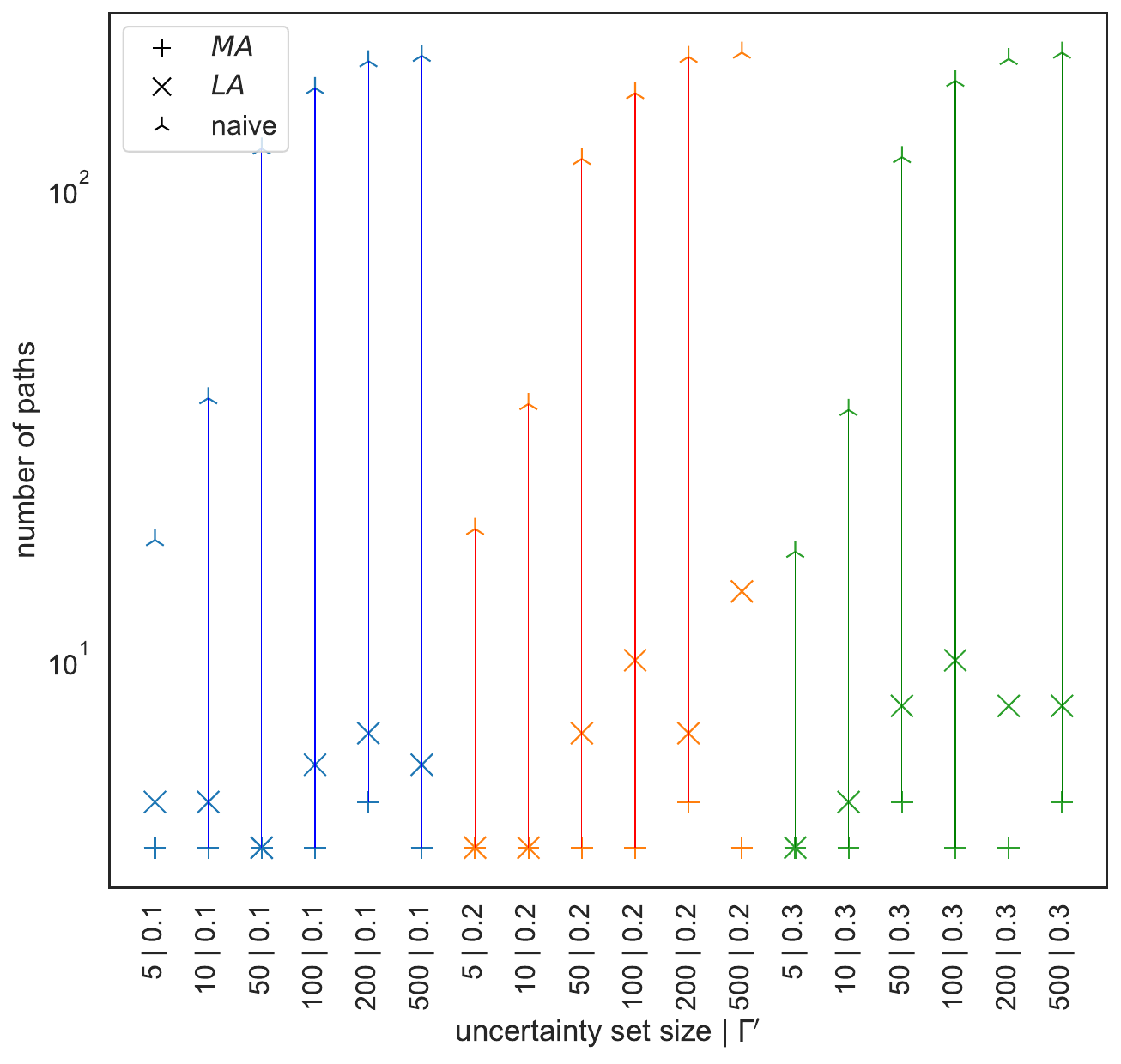}
        }
        \caption{Values of $\mathcal{Y}$ (logarithmic scale)}
        \label{fig:res_mouse_y}
\end{subfigure}
\begin{subfigure}{0.49\linewidth}
	\resizebox{1\textwidth}{!}{
\includegraphics[width=1\linewidth]{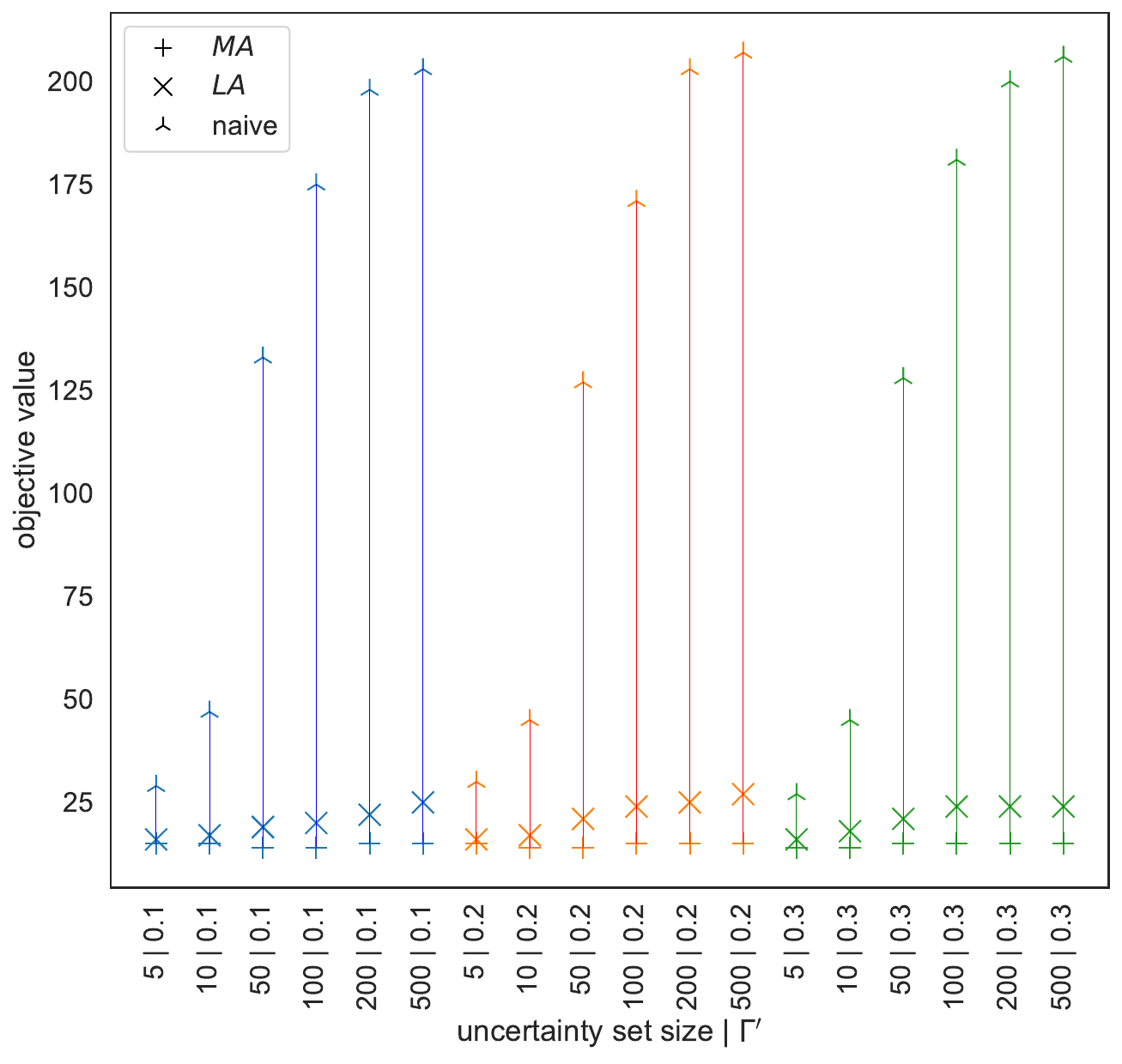}
        }
        \caption{Objective values}
        \label{fig:res_mouse_obj}
		\end{subfigure}
\end{center}
	\caption{Results of \MA, \LA{} and the naive approach for instance set \emph{mouse} for the different uncertainty set sizes $ |\mathcal{U}(\Gamma)|$ and values for $\Gamma'$.}
	\label{fig:res_mouse}
	\end{figure}	

 \begin{figure}[h!]
	\begin{center}
        \begin{subfigure}{0.49\linewidth}
	   \resizebox{1\textwidth}{!}{
        \includegraphics[width=1\linewidth]{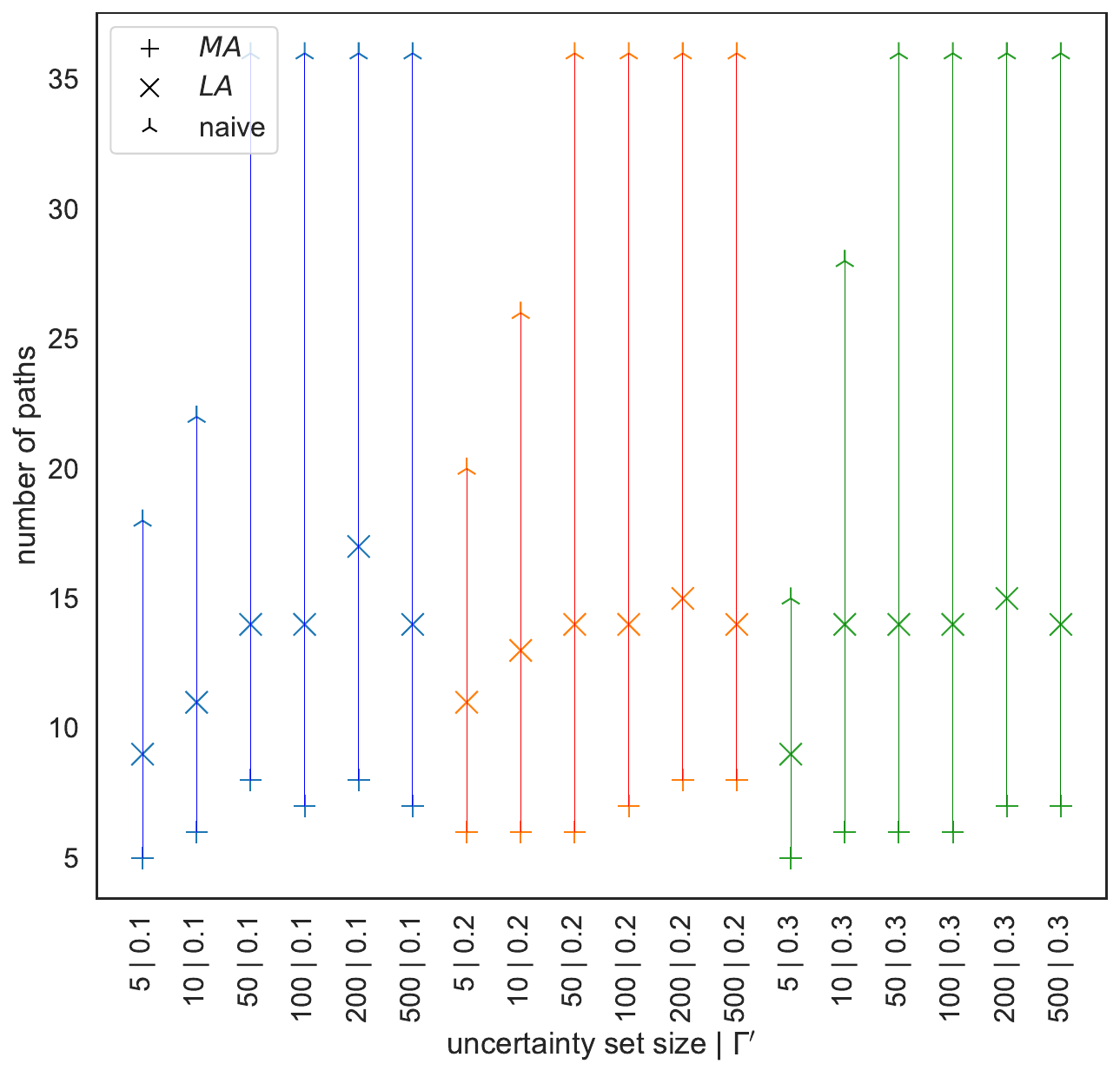}
        }
        \caption{Values of $\mathcal{Y}$}
        \label{fig:res_gte_y}
\end{subfigure}
\begin{subfigure}{0.49\linewidth}
	\resizebox{1\textwidth}{!}{
\includegraphics[width=1\linewidth]{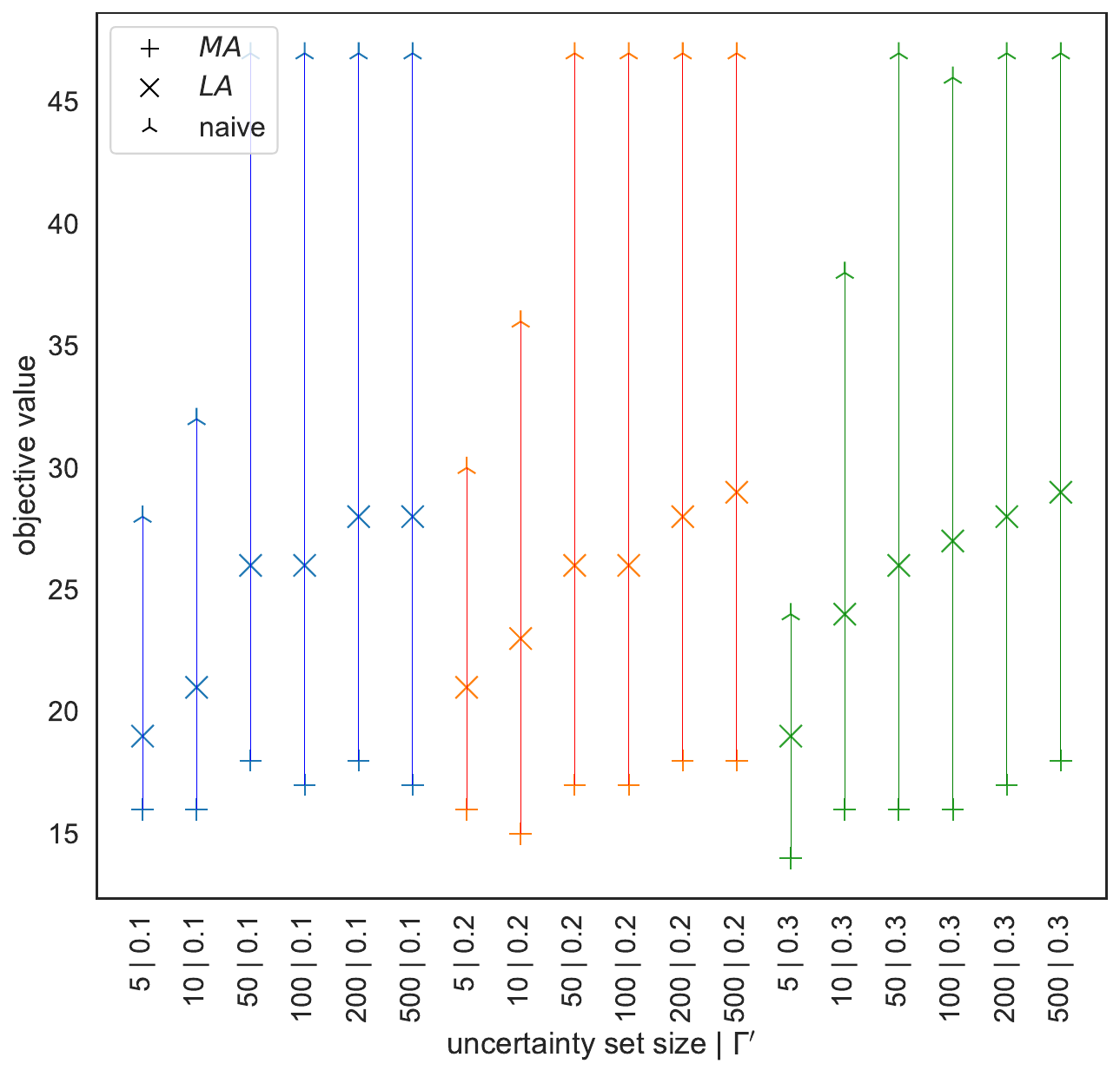}
        }
        \caption{Objective values}
        \label{fig:res_gte_obj}
		\end{subfigure}
\end{center}
	\caption{Results of \MA, \LA{} and the naive approach for instance set \emph{gte} for the different uncertainty set sizes $ |\mathcal{U}(\Gamma)|$ and values for $\Gamma'$.}
	\label{fig:res_gte}
	\end{figure}

The values for $\mathcal{Y}$ for \emph{lowersaxony} and \emph{human} are shown in Figures \ref{fig:res_saxy} and \ref{fig:res_humany}, respectively.

\begin{figure}[h!]
	\begin{center}
        \begin{subfigure}{0.49\linewidth}
	   \resizebox{1\textwidth}{!}{
        \includegraphics[width=1\linewidth]{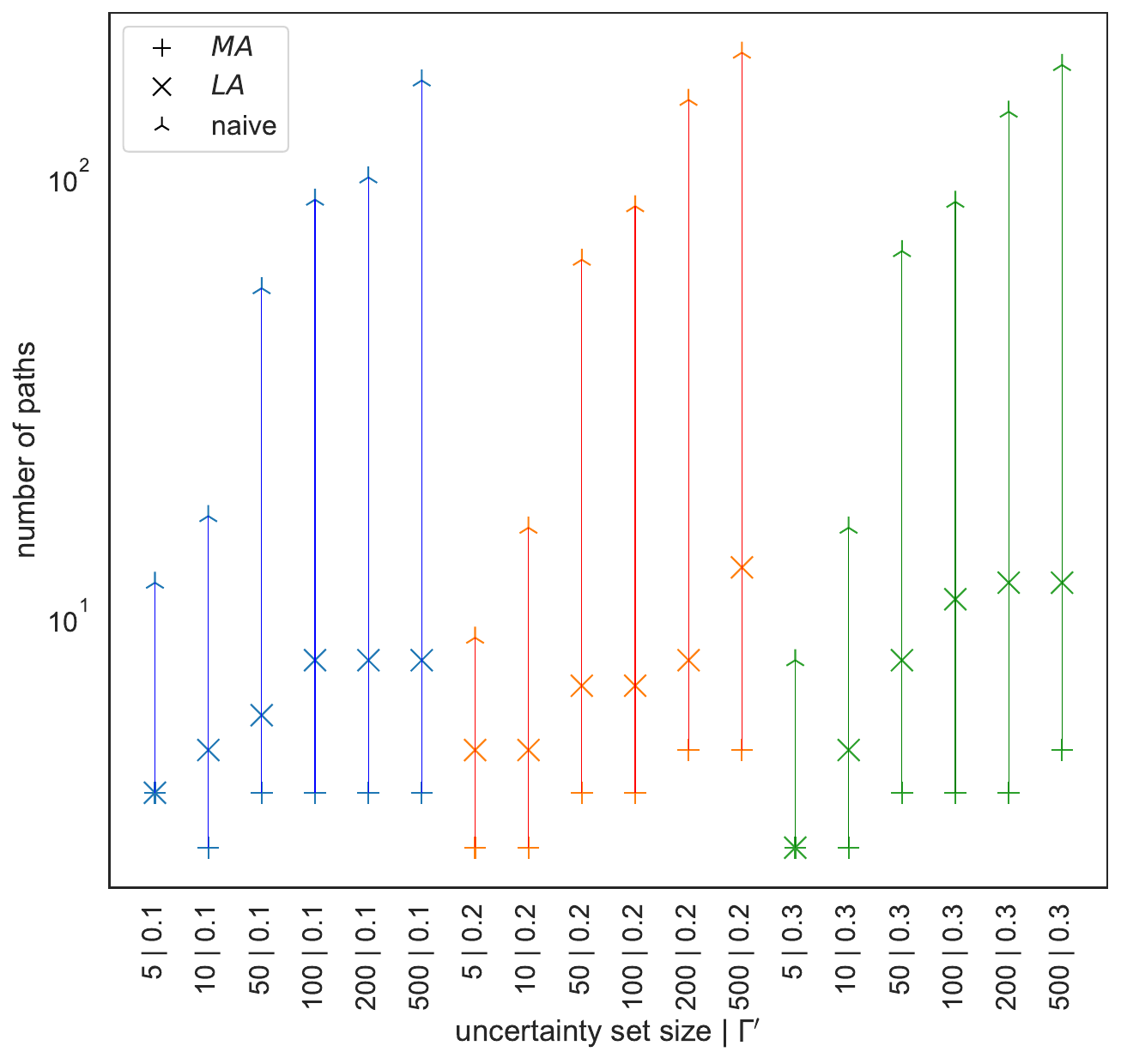}
        }
        \caption{Results for instance set \emph{lowersaxony}}
        \label{fig:res_saxy}
\end{subfigure}
\begin{subfigure}{0.49\linewidth}
	\resizebox{1\textwidth}{!}{
\includegraphics[width=1\linewidth]{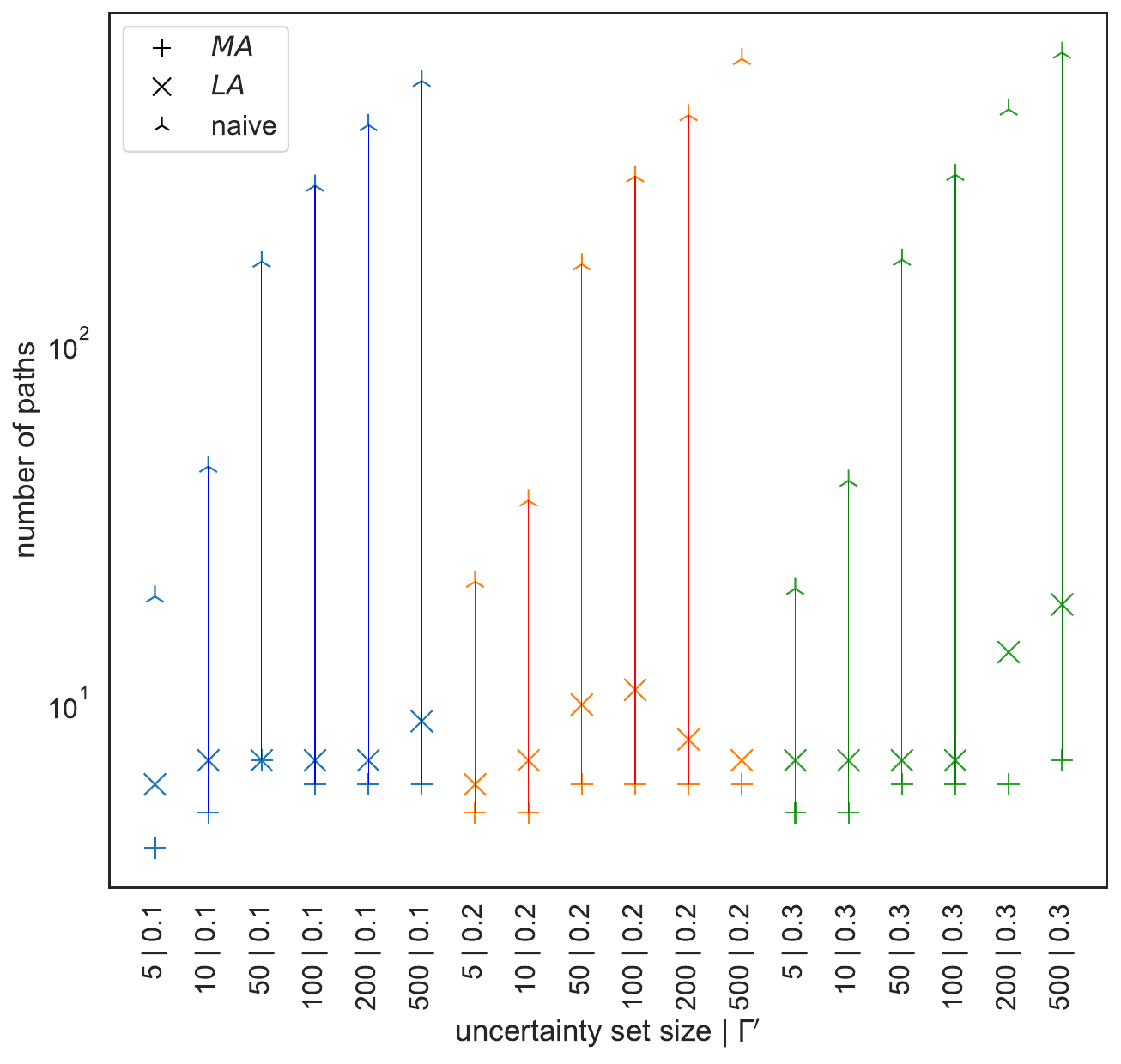}
        }
        \caption{Results for instance set \emph{human}}
        \label{fig:res_humany}
		\end{subfigure}
\end{center}
	\caption{Values for $\mathcal{Y}$ of \MA, \LA{} and the naive approach for the different uncertainty set sizes $ |\mathcal{U}(\Gamma)|$ and values for $\Gamma'$ (logarithmic scale).}
	\label{fig:res_yy}
	\end{figure}

\end{document}